\documentclass[smallextended]{svjour3-ppn}

\smartqed
\usepackage{amsmath,amsfonts,amssymb,amsxtra,setspace,xspace,graphicx,lmodern}
\usepackage[colorlinks=true]{hyperref}
\hypersetup{urlcolor=blue, citecolor=red, linkcolor=blue}
\usepackage[usenames,dvipsnames]{color}

\newtheorem{thm}{Theorem}
\newtheorem{cor}[thm]{Corollary}
\newtheorem{lem}[thm]{Lemma}
\newtheorem{pro}[thm]{Proposition}

\newtheorem{oss}[thm]{Remark}
\numberwithin{thm}{section}
\numberwithin{equation}{section}

\newcommand{\C}{{\mathsf C}}
\newcommand{\D}{{\mathcal D(\R^d)}}
\newcommand{\R}{{\mathbb R}}
\newcommand{\RF}{{\mathcal G}}
\newcommand{\N}{{\mathbb N}}
\newcommand{\be}[1]{\begin{equation}\label{#1}}
\newcommand{\ee}{\end{equation}}
\renewcommand{\(}{\left(}
\renewcommand{\)}{\right)}
\newcommand{\iRd}[1]{\int_{\R^d}{#1}\,dx}
\newcommand{\nrm}[2]{\left\|{#1}\right\|_{#2}}
\newcommand{\Lqgamma}{\mathrm L^q_\gamma(\R^d)}
\newcommand{\Hpgamma}{\mathcal H_{p,\gamma}(\R^d)}

\newcommand{\Hpgammas}{\mathcal H_{p,\gamma}^\star}

\journalname{Mathematische Annalen (Giga)}

\begin{document}\title{Weighted interpolation inequalities: a perturbation approach}
\titlerunning{Weighted interpolation inequalities}

\author{Jean Dolbeault \and Matteo Muratori \and Bruno Nazaret}
\authorrunning{J.~Dolbeault \and M.~Muratori \and B.~Nazaret}
\institute{J.~Dolbeault \at Ceremade, CNRS UMR n$^{\circ}$ 7534 and Universit\'e Paris-Dauphine, PSL\kern-0.5pt\raisebox{2.3pt}{$\star$}, Place de Lattre de Tassigny, 75775 Paris C\'edex~16, France. \email{\href{mailto:dolbeaul@ceremade.dauphine.fr}{dolbeaul@ceremade.dauphine.fr}}
\and M.~Muratori \at Dipartimento di Matematica \emph{Francesco Brioschi}, Politecnico di Milano, Piazza Leonardo da Vinci 32, 20133 Milano, Italy. \email{\href{mailto:matteo.muratori@polimi.it}{matteo.muratori@polimi.it}}
\and B.~Nazaret \at SAMM, Universit\'e Paris 1, 90, rue de Tolbiac, 75634 Paris C\'edex~13, France.\\ \email{\href{mailto:Bruno.Nazaret@univ-paris1.fr}{Bruno.Nazaret@univ-paris1.fr}}}
\date{\today}
\maketitle

%%%%%%%%%%%%%%%%%%%%%%%%%%%%%%%%%%%%%%%%%%%%%%%%%%%%%%%%%%%%%%%%%%%%%%
\begin{abstract}
We study optimal functions in a family of Caffarelli-Kohn-Niren\-berg inequalities with a power-law weight, in a regime for which standard symmetrization techniques fail. We establish the existence of optimal functions, study their properties and prove that they are radial when the power in the weight is small enough. Radial symmetry up to translations is true for the limiting case where the weight vanishes, a case which corresponds to a well-known subfamily of Gagliardo-Nirenberg inequalities. Our approach is based on a concentration-compactness analysis and on a perturbation method which uses a spectral gap inequality. As a consequence, we prove that optimal functions are explicit and given by Barenblatt-type profiles in the perturbative regime.


%%%%%%%%%%%%%%%%%%%%%%%%%%%%%%%%%%%%%%%%%%%%%%%%%%%%%%%%%%%%%%%%%%%%%%
\end{abstract}

%%%%%%%%%%%%%%%%%%%%%%%%%%%%%%%%%%%%%%%%%%%%%%%%%%%%%%%%%%%%%%%%%%%%%%
%%%%%%%%%%%%%%%%%%%%%%%%%%%%%%%%%%%%%%%%%%%%%%%%%%%%%%%%%%%%%%%%%%%%%%
\section{Introduction}\label{intro}

This paper is devoted to a special class of \emph{Caffarelli-Kohn-Nirenberg interpolation inequalities} that were introduced in~\cite{Caffarelli-Kohn-Nirenberg-84} and can be written as
\be{CKN}
\nrm w{2p,\gamma}\le\C_\gamma\,\nrm{\nabla w}2^\vartheta\,\nrm w{p+1,\gamma}^{1-\vartheta}\quad\forall\,w\in\D\,,
\ee
where $\D$ denotes the space of smooth functions on $\R^d$ with compact support, $\C_\gamma$ is the \emph{best constant} in the inequality,
\be{eq: parameters}
d\ge3\,,\quad\gamma\in(0,2)\,,\quad p\in\(1,2^\ast_\gamma/2\)\quad\mbox{with}\quad 2^\ast_\gamma:=2\,\frac{d-\gamma}{d-2}
\ee
and
\be{eq: theta-2}
\vartheta:=\frac{2^\ast_\gamma\,(p-1)}{2\,p\(2^\ast_\gamma-p-1\)}=\frac{(d-\gamma)\,(p-1)}{p\,\big(d+2-2\,\gamma-p\,(d-2)\big)}\,.
\ee
The norms are defined by
\[
\nrm w{q,\gamma}:=\(\iRd{|w|^q\,|x|^{-\gamma}}\)^{1/q}\quad\mbox{and}\quad\nrm wq:=\nrm w{q,0}\,.
\]
The optimal constant $\C_\gamma$ is determined by the minimization of the quotient
\[
\mathcal Q_\gamma[w]:=\frac{\nrm{\nabla w}2^\vartheta\,\nrm w{p+1,\gamma}^{1-\vartheta}}{\nrm w{2p,\gamma}}\,.
\]

When $\gamma=0$, Inequalities~\eqref{CKN} become a particular subfamily of the well-known Gagliardo-Nirenberg inequalities introduced in~\cite{Ga,Nir}. In that case, optimal functions have been completely characterized in~\cite{DD}. Gagliardo-Nirenberg inequalities have attracted lots of interest in the recent years: see for instance~\cite{DT2011} and references therein, or~\cite{MR3155209}.

\medskip We have two main reasons to consider such a problem. First of all, optimality \emph{among radial functions} is achieved by
\be{Barenblatt-w}
w_\star(x):=\(1+|x|^{2-\gamma}\)^{-\frac1{p-1}}\quad\forall\,x\in\R^d
\ee
up to multiplications by a constant and scalings as we shall see later. It is remarkable that the function $w_\star$ clearly departs from standard optimal functions that are usually characterized using the conformal invariance properties of the sphere and the stereographic projection, like for instance in~\cite[Section~6.10]{MR3155209}.

If $d=1$, it is elementary to prove that optimal functions for~\eqref{CKN} are of the form~\eqref{Barenblatt-w} up to multiplication by constants and scalings. The case $d=2$ is not considered in this paper. In any higher dimension $d\ge3$, even with a radial weight of the form $|x|^{-\gamma}$, there is no simple symmetry result that would allow us to identify the optimal functions in terms of $w_\star$. In other words, it is not known if equality holds in
\be{Ineq:Symmetry}
\inf_{w\in\mathcal D^\star(\R^d)\setminus\{0\}}\mathcal Q_\gamma[w]=:\(\C^\star_\gamma\)^{-1}=\mathcal Q_\gamma\left[w_\star\right]\ge\(\C_\gamma\)^{-1}:=\inf_{w\in\mathcal D(\R^d)\setminus\{0\}}\mathcal Q_\gamma[w]\,,
\ee
where $\mathcal D^\star(\R^d)$ denotes the subset of $\mathcal D(\R^d)$ which is spanned by radial smooth functions, \emph{i.e.},~smooth functions which depend only on $|x|$. Our main result is a first step in this direction.
%---------------------------------------------------------------------
\begin{thm}\label{thm: main-thm-barenblatt-gamma-intro} Let $d\ge3$. For any $p\in(1,d/(d-2))$, there exists a positive $\gamma^\ast$ such that equality holds in~\eqref{Ineq:Symmetry} for all $\gamma\in(0,\gamma^\ast)$.\end{thm}
%---------------------------------------------------------------------
A slightly stronger result is given in Theorem~\ref{thm: main-thm-barenblatt-gamma}.

We remark that optimal functions for $\mathcal Q_\gamma$ can be assumed to be nonnegative and satisfy, up to multiplications by a constant and scalings, the semilinear equation
\be{EL}
-\,\Delta w+|x|^{-\gamma}\big(w^p-w^{2p-1}\big)=0\,.
\ee
This will be discussed in Section~\ref{sec: prelim}. However, the classical result of B.~Gidas, W.M.~Ni and L.~Nirenberg in~\cite{MR634248} does not allow us to decide if a positive solution of~\eqref{EL} has to be radially symmetric. So far, it is not known yet if the result can be deduced from a symmetrization method either, even for a minimizer of $\mathcal Q_\gamma$. We shall say that \emph{symmetry breaking} occurs if $\C^\star_\gamma<\C_\gamma$. Whether this happens for some $\gamma\in(0,2)$ and $p$ in the appropriate range, or not, is an open question.

\medskip The symmetry result of Theorem~\ref{thm: main-thm-barenblatt-gamma-intro} has very interesting consequences, and here is a second motivation for this paper. Let us consider the \emph{fast diffusion equation with weight}
\be{Eqn:FD}
u_t+|x|^\gamma\,\nabla\cdot\(u\,\nabla u^{m-1}\)=0\,,\quad(t,x)\in\R^+\times\R^d\,,
\ee
with initial condition $u(t=0,\cdot)=u_0\in\mathrm L^1(\R^d,|x|^{-\gamma}\,dx)$, $u_0\ge0$ and $m\in(m_1,1)$, where
\[
m_1:=\frac{2\,d-\gamma-2}{2\,(d-\gamma)}\,.
\]
{}From the point of view of the well-posedness of the Cauchy problem and of the long-time behaviour of the solutions, such an equation, in the porous media case (namely for $m>1 $), has been studied in~\cite{MR2425009,MR2461559}. In this paper, we consider the fast diffusion regime (case $m<1$). In particular, it can be shown that the mass $M:=\iRd{u\,|x|^{-\gamma}}$ is independent of $t$. Let us introduce the time-dependent rescaling
\be{Eqn:ChangeOfVariables}
u(t,x)=R^{\gamma-d}\,v\((2-\gamma)^{-1}\log R\,,\,\frac xR\),
\ee
with $R=R(t)$ defined by
\[
\frac{dR}{dt}=(2-\gamma)\,R^{(m-1)(\gamma-d)-1+\gamma},\quad R(0)=1\,.
\]
The solution is explicit and given by
\[\label{Eqn:EDO}
R(t)=\big[\,1+(2-\gamma)\,(d-\gamma)\,(m-m_c)\,t\,\big]^\frac1{(d-\gamma)\,(m-m_c)}\quad\mbox{where}\quad m_c:=\frac{d-2}{d-\gamma}\,.
\]
After changing variables we obtain that the rescaled function $v$ solves the Fokker-Planck-type equation
\be{Eqn:FD-FP}
v_t+|x|^\gamma\,\nabla\cdot\Big[\,v\,\nabla\(v^{m-1}-|x|^{2-\gamma}\)\Big]=0
\ee
with initial condition $v(t=0,\cdot)=u_0$. The convergence of the solution of~\eqref{Eqn:FD} towards a self-similar solution of Barenblatt type as time goes to $\infty$ is replaced by the convergence of $v$ towards a stationary solution of~\eqref{Eqn:FD-FP} given by
\[
\mathfrak B(x):=\(C+|x|^{2-\gamma}\)^\frac 1{m-1}\,,
\]
where $C>0$ is uniquely determined by the condition
\[
\iRd{\mathfrak B\,|x|^{-\gamma}}=M\,.
\]
A straightforward computation shows that the \emph{free energy}, or \emph{relative entropy}
\[
\mathcal F[v]:=\frac1{m-1}\iRd{\(v^m-\mathfrak B^m-\,m\,\mathfrak B^{m-1}\,(v-\mathfrak B)\)|x|^{-\gamma}}
\]
is nonnegative and satisfies
\be{Eqn:FI}
\frac d{dt}\mathcal F[v(t,\cdot)]=-\,\mathcal I[v(t,\cdot)]\,,
\ee
where the \emph{Fisher information} is defined by
\[
\mathcal I[v]:=\frac m{1-m}\iRd{v\left|\,\nabla v^{m-1}-(2-\gamma)\,\frac x{|x|^\gamma}\right|^2}\,.
\]

As a consequence of Theorem~\ref{thm: main-thm-barenblatt-gamma-intro}, an elementary computation shows that the \emph{entropy -- entropy production inequality}
\be{CKN2}
(2-\gamma)^2\,\mathcal F[v]\leq\mathcal I[v]
\ee
 holds if $\gamma\in(0,\gamma^\ast)$. More precisely,~\eqref{CKN2} is equivalent to~\eqref{CKN} if we take $w=v^{m-1/2}$, $p=1/(2\,m-1)$ and perform a scaling. Accordingly, notice that $\mathfrak B^{m-1/2}$ is equal to $w_\star$ up to a scaling and a multiplication by a constant. This generalizes to $\gamma\in(0,\gamma^\ast)$ the results obtained in~\cite{DD} for the case~\hbox{$\gamma=0$}. Whether $(2-\gamma)^2$ is the best constant in~\eqref{CKN2} is a very natural question. The answer is \emph{yes} if $\gamma=0$ and \emph{no} if $\gamma>0$, as long as \emph{symmetry} (no symmetry breaking) holds. An answer has been provided in~\cite[Proposition~11]{2016arXiv160208319B}, which relies on considerations on the linearization. Key ideas are provided by the weighted fast diffusion equation whose large time asymptotics are governed by the linearized problem. These asymptotics are studied in~\cite{2016arXiv160208315B}. Recent progresses on the issue of symmetry breaking, which partially rely on the present paper, have been achieved in~\cite{2016arXiv160506373D}.
 
A consequence of~\eqref{CKN2} is the exponential convergence of the solution $v$ of~\eqref{Eqn:FD-FP}~to~$\mathfrak B$.
%---------------------------------------------------------------------
\begin{cor}\label{cor:rateFD} Let $d\ge3$, $m\in(1-1/d,1)$ and $\gamma\in(0,\gamma^\ast)$, where $\gamma^\ast$ is defined as in Theorem~\ref{thm: main-thm-barenblatt-gamma-intro} for $p=1/(2\,m-1)$. If $v$ is a solution to~\eqref{Eqn:FD-FP} with nonnegative initial datum $u_0$ such that $u_0\in\mathrm L^1(\R^d,|x|^{-\gamma}\,dx)$ and $\iRd{u_0^m\,|x|^{-\gamma}}+\iRd{u_0\,|x|^{2-2\,\gamma}}$ is finite, then
\[
\mathcal F[v(t,\cdot)]\le\mathcal F[u_0]\;e^{-\,(2-\gamma)^2\,t}\quad\forall\,t>0\,.
\]
\end{cor}
%---------------------------------------------------------------------
The above exponential decay is actually equivalent to~\eqref{CKN2} and henceforth to~\eqref{CKN} with $\C_\gamma=\C^\star_\gamma$ as can be checked by computing $\frac d{dt}\mathcal F[v(t,\cdot)]$ at $t=0$. The free energy $\mathcal F[v]$ is a measure of the distance between $v$ and $\mathfrak B$. Exactly as in the case $\gamma=0$, one can undo the change of variables~\eqref{Eqn:ChangeOfVariables} and write an intermediate asymptotics result based on the Csisz\'ar-Kullback inequality. The method is somewhat classical and will not be developed further in this paper. See for instance~\cite{DD,DT2011} for more details. To prove Corollary~\ref{cor:rateFD}, one has to show that the mass $M$ is conserved along the flow defined by~\eqref{Eqn:FD-FP} and that~\eqref{Eqn:FI} holds: this can be done as in~\cite{BBDGV} when $\gamma=0$.

\smallskip
Before entering in the details, let us mention that the case of the porous media equation with $m>1$ has been more studied than the fast diffusion case $m<1$. We refer the reader \emph{e.g.}~to \cite{GMP,GM13,GM14} for some recent results relating suitable functional inequalities with the asymptotic properties of the solutions.

\medskip Let us introduce some basic notations for the functional spaces. We define the spaces $\Lqgamma$ and $\Hpgamma$, respectively, as the space of all measurable functions $w$ such that $\nrm w{q,\gamma}$ is finite and the space of all measurable functions $w$, with $\nabla w$ measurable, such that $\|w\|_{\Hpgamma}:=\nrm w{p+1,\gamma}+\nrm{\nabla w}2$ is finite. Moreover, we denote as $\Hpgammas(\R^d)$ the subspace of $\Hpgamma$ spanned by radial functions. A simple density argument shows, in particular, that $\D$ is dense in $\Hpgamma$ (Lemma~\ref{prop: density}), so that inequality~\eqref{CKN} holds for any function in $\Hpgamma$ and, as a consequence, $\Hpgamma$ is continuously embedded into $\mathrm L^{2p}_\gamma(\R^d)$.

\medskip The proof of Theorem~\ref{thm: main-thm-barenblatt-gamma-intro} is based on a perturbation method which relies on the fact that, by the results in~\cite{DD}, the optimal functions in the case $\gamma=0$ are radial up to translations. Our strategy is adapted from~\cite{DELT09}, except that we have no Emden-Fowler type transformation that would allow us to get rid of the weights. This has the unpleasant consequence that a fully developed analysis of the convergence of any optimal function for~\eqref{CKN} is needed, based on a concentration-compactness method, as $\gamma\downarrow 0$. We prove that the limit is the radial solution, namely the only one centered at the origin, to the limit problem, although the limit problem is translation invariant. Then we are able to prove that the optimal functions are themselves radially symmetric for $\gamma>0$ sufficiently small. As a consequence, $\C^{\star}_\gamma=\C_\gamma$ for any such~$\gamma$, and the optimal functions are all given by~\eqref{Barenblatt-w} up to a multiplication by a constant and a scaling. Here are the main steps of our approach.

\begin{enumerate}

\item We work with the non-scale-invariant form of~\eqref{CKN} that can be written as
\be{eq: E-funz-energy}
\mathcal E_\gamma[w]:=\frac12\,\nrm{\nabla w}2^2+\frac1{p+1}\,\nrm w{p+1,\gamma}^{p+1}-\mathsf J_\gamma\,\nrm w{2p,\gamma}^{2p\,\theta_\gamma}\ge0\,,
\ee
where $\mathsf J_\gamma$ denotes the optimal constant and
\be{eq: first-def-theta-gamma}
\theta_\gamma:=\frac{d+2-2\,\gamma-p\,(d-2)}{d-\gamma-p\,(d+\gamma-4)}\in(0,1)\,.
\ee
A simple scaling argument given in Section~\ref{sec: prelim} shows that~\eqref{CKN} and~\eqref{eq: E-funz-energy} are equivalent, and relates~$\mathsf J_\gamma$ and~$\C_\gamma$: see~\eqref{I-C}.

\item In Section~\ref{sec: prelim} we establish the compactness of the embedding of $\Hpgamma$ into $\mathrm L^{2p}_\gamma(\R^d)$, which implies the existence of at least one optimal function~$w_\gamma$ for~\eqref{eq: E-funz-energy}. This optimal function solves~\eqref{EL} up to a multiplication by a constant and a scaling. Notice that optimal functions for~\eqref{eq: E-funz-energy} are not necessarily unique, even up to multiplication by constants. For simplicity we shall pick one optimal function for each $\gamma>0$, denote it by $w_\gamma$, but each time we use this notation, one has to keep in mind that it is not \emph{a priori} granted that $w_\gamma$ is uniquely defined. We also adopt the convention that $w_0$ denotes the unique \emph{radial} minimizer, having a suitably prescribed $\mathrm L^{2p}$ norm, corresponding to $\gamma=0$ (see~\cite{DD} for details). Up to a multiplication by a constant and a scaling, $w_0$ is equal to $w_\star$ given by~\eqref{Barenblatt-w}, with $\gamma=0$.

Next, in Section~\ref{sect: stime-apriori}, we prove integrability and regularity estimates for solutions to~\eqref{EL}. We point out that $C^{1,\alpha}$ regularity can be expected only if $\gamma\in(0,1)$, as it can be easily guessed by considering the function $w_\star$, which involves $|x|^{2-\gamma}$ (see Remark~\ref{rem: regularity}).

\item Inequality~\eqref{eq: E-funz-energy} means $\mathcal E_\gamma[w]\ge\mathcal E_\gamma[w_\gamma]=0=\mathcal E_0[w_0]$. Hence we know that
\[\label{eq: E-div-gamma}
\mathcal E_\gamma[w_\gamma]-\mathcal E_0[w_\gamma]\le0\le\mathcal E_\gamma[w_0]-\mathcal E_0[w_0]\,.
\]
The concentration-compactness analysis of Section~\ref{sec: conc-comp} shows that, up to the extraction of subsequences, $\lim_{\gamma\to0}w_\gamma(\cdot+y_\gamma)=w_0(\cdot+y_0)$, where $y_\gamma\in\R^d$ is a suitable translation. By passing to the limit as $\gamma\downarrow0$, we obtain
\[\label{eq: E-div-gamma-lim}
\limsup_{\gamma\to0}\iRd{\(\tfrac1{2p}\,w_\gamma^{2p}\,|x|^{-\gamma}-\tfrac1{p+1}\,w_\gamma^{p+1}\)\log|x|}\le\lim_{\gamma\to0}\frac{\mathcal E_\gamma[{w}_0]-\mathcal E_0[{w}_0]}\gamma\,.
\]
If the r.h.s.~was explicit, finite, this would allow us to deduce that $\{y_\gamma\}$ is bounded. This is not the case because $\mathsf J_\gamma$ appears in the expression of~$\mathcal E_\gamma$. To circumvent this difficulty, we can use a rescaled Barenblatt-type function in place of $w_0$ and get an equivalent formulation in which the r.h.s.~stays bounded. This is done in~Section~\ref{sec: bnd-transl}.
\item Inspired by selection principles in Gamma-convergence methods as in~\cite{AB}, we infer that $y_0$ minimizes the function
\[
y \mapsto \iRd{\(-\frac{w_0^{p+1}}{p+1}+\frac{w_0^{2p}}{2\,p}\)\log|x+y|}\,.
\]
The minimum turns out to be attained exactly at $y=0$, so that $\{w_\gamma\}$ converges to $w_0$. The detailed analysis is carried out in the proof of Proposition~\ref{lem: minimizer}.

\item We prove Theorem~\ref{thm: main-thm-barenblatt-gamma-intro} by contradiction in Section~\ref{sec: radiality}, using the method of~\cite{MR2437030,DELT09}. \emph{Angular derivatives} of $w_\gamma$ are nontrivial if $w_\gamma$ is not radial. By differentiating~\eqref{EL}, one finds that the angular derivatives of $w_\gamma$ belong to the kernel of a suitable operator. Passing to the limit as $\gamma\downarrow0$, we get a contradiction with a spectral gap property of the limit operator.

\end{enumerate}

\medskip Inequality~\eqref{CKN} is a special case of the \emph{Caffarelli-Kohn-Nirenberg} inequalities. Because these inequalities involve weights, symmetry and symmetry breaking are key issues. However, only special cases have been studied so far. We refer to~\cite{Oslo} for a review, to~\cite{FreefemDolbeaultEsteban,0951-7715-27-3-435} for some additional numerical investigations, and to~\cite{1406,DEL2015} for more recent results. Concerning the existence of optimal functions in the \emph{Hardy-Sobolev} case $p=(d-\gamma)/(d-2)$, the reader is invited to read~\cite{Catrina-Wang-01,1005}. We may observe that inequality~\eqref{CKN} has three endpoints for which symmetry is known: the case of the \emph{Gagliardo-Nirenberg inequalities} corresponding to $\gamma=0$, the case of the \emph{Hardy-Sobolev inequalities} with $p=(d-\gamma)/(d-2)$ and $\vartheta=1$ and, as a special case, the \emph{Hardy inequality} with $(p,\gamma)=(1,2)$: see for instance~\cite{DELT09}. No other results specific of~\eqref{CKN} are known so far.

Symmetry issues are difficult problems. In most of the cases, \emph{symmetry breaking} is proved using a spherical harmonics expansion as in~\cite{Catrina-Wang-01,Felli-Schneider-03,DDFT} and linear instability, although an energy method has also been used in \cite{DETT}. For \emph{symmetry}, there is a variety of methods which, however, cover only special cases. Moving plane methods as in~\cite{Chou-Chu-93} or symmetrization techniques like in~\cite{Hor97,MR1734159,DETT} can be applied to establish that the optimal functions are given by~\eqref{Barenblatt-w}, up to multiplications by a constant and scalings, in a certain range of the parameters. Symmetry has also been proved by direct estimates, \emph{e.g.}, in~\cite{DETT,DEL2011}, and recently in~\cite{DEL2015} using rigidity estimates based on heuristics arising from entropy methods in nonlinear diffusion equations. Beyond the range covered by symmetrization and moving planes techniques, the best established method relies on \emph{perturbation techniques} that have been used in~\cite{MR2001882,Lin-Wang-04,MR2437030,DELT09}. In the present paper, we shall argue by perturbation, with new difficulties due to the translation invariance of the limiting problem.

%%%%%%%%%%%%%%%%%%%%%%%%%%%%%%%%%%%%%%%%%%%%%%%%%%%%%%%%%%%%%%%%%%%%%%
%%%%%%%%%%%%%%%%%%%%%%%%%%%%%%%%%%%%%%%%%%%%%%%%%%%%%%%%%%%%%%%%%%%%%%
\section{Preliminary results}\label{sec: prelim}

%%%%%%%%%%%%%%%%%%%%%%%%%%%%%%%%%%%%%%%%%%%%%%%%%%%%%%%%%%%%%%%%%%%%%%
\subsection{Interpolation}

We denote by $\dot{\mathrm H}^1(\R^d)$ the closure of $\mathcal{D}(\R^d)$ w.r.t.~the norm $w\mapsto\nrm{\nabla w}2$. Assume that $d\ge3$. It is well known that for all $\gamma\in[0,2]$ there exists a positive constant $\mathsf C_{\rm HS}$ such that the Hardy-Sobolev inequality
\be{eq: sobolev-hardy-H}
\nrm w{2^\ast_\gamma,\gamma}\le\mathsf C_{\rm HS}\,\nrm{\nabla w}2\quad\forall\,w\in\dot{\mathrm H}^1(\R^d)
\ee
holds, where the exponent $2^\ast_\gamma$ has been defined in~\eqref{eq: parameters}. Let $2^\ast=2^\ast_0=\frac{2\,d}{d-2}$. For $\gamma=0$ and $\gamma=2$, we recover respectively the Sobolev inequality
\be{eq: sobolev-H}
\nrm w{2^\ast}\le\mathsf C_{\rm S}\,\nrm{\nabla w}2\quad\forall\,w\in\dot{\mathrm H}^1(\R^d)
\ee
and the Hardy inequality
\be{eq: hardy-H}
\nrm w{2,2}\le\mathsf C_{\rm H}\,\nrm{\nabla w}2\quad\forall\,w\in\dot{\mathrm H}^1(\R^d)\,.
\ee
A H\"older interpolation shows that $\nrm w{2^\ast_\gamma,\gamma}\le\nrm w{2,2}^{\frac\gamma2\frac{d-2}{d-\gamma}}\,\nrm w{2^\ast}^{\frac d2\frac{2-\gamma}{d-\gamma}}$ and hence
\[
\mathsf C_{\rm HS}\le\mathsf C_{\rm H}^{\frac\gamma2\frac{d-2}{d-\gamma}}\,\mathsf C_{\rm S}^{\frac d2\frac{2-\gamma}{d-\gamma}}
\]
for any $\gamma\in[0,2]$. The best constant in~\eqref{eq: sobolev-H} has been identified in~\cite{Aub,Tal} and it is well known that $\mathsf C_{\rm H}=2/(d-2)$. According to~\cite{Chou-Chu-93,Hor97,DELT09}, symmetry holds so that $\mathsf C_{\rm HS}$ is easy to compute using the optimal function $w_\star$ defined by~\eqref{Barenblatt-w} with $p=2^\ast_\gamma/2=(d-\gamma)/(d-2)$, for any $\gamma\in(0,2]$. The H\"older interpolation $\nrm w{2p,\gamma}\le\nrm w{2^\ast_\gamma,\gamma}^\vartheta\,\nrm w{p+1,\gamma}^{1-\vartheta}$ with $\vartheta$ as in~\eqref{eq: theta-2} shows that the Caffarelli-Kohn-Nirenberg inequality~\eqref{CKN} holds with
\be{eq: CKN-originaria-interp}
\C_\gamma\le\big(\mathsf C_{\rm HS}\big)^\vartheta\,.
\ee

%%%%%%%%%%%%%%%%%%%%%%%%%%%%%%%%%%%%%%%%%%%%%%%%%%%%%%%%%%%%%%%%%%%%%%
\subsection{Scalings and Euler-Lagrange equation}\label{Sec: Scalings and EL}
Consider the following functional:
\[\label{eq: funcQ}
\RF_\gamma[w]:=\frac12\,\nrm{\nabla w}2^2+\frac1{p+1}\,\nrm w{p+1,\gamma}^{p+1}\quad\forall\,w\in\Hpgamma\,.
\]
Inequality~\eqref{eq: E-funz-energy} amounts to
\[
\mathsf J_\gamma=\frac1{M^{\theta_\gamma}}\,\inf\Big\{\RF_\gamma[w]\,:\,w\in\Hpgamma\,,\;\nrm w{2p,\gamma}^{2p}=M\Big\}
\]
for any $M>0$ and we shall consider the problem of the existence of an optimal function $w_\gamma$, that is, the existence of $w_\gamma$ such that
\be{prob: min}
\RF_\gamma[w_\gamma]=\mathsf J_\gamma\,M^{\theta_\gamma}\,,\quad\,w_\gamma\in\Hpgamma\,,\quad\nrm{w_\gamma}{2p,\gamma}^{2p}=M\,,
\ee
where $\theta_\gamma$ is defined by~\eqref{eq: first-def-theta-gamma}. Let us check that $\mathsf J_\gamma$ is independent of $M$. The scaling defined by
\[
w^\lambda(x):=\lambda^\frac{d-\gamma}{2\,p}\,w(\lambda\,x)\quad\forall\,x\in\R^d\,, \quad \textrm{with} \lambda>0\,,
\]
is such that $\nrm{w^\lambda}{2p,\gamma}=\nrm w{2p,\gamma}$, while a change of variables gives
\[
\RF_\gamma[w^\lambda]=\frac 12\,\lambda^{\frac{d-\gamma}p-(d-2)}\,\nrm{\nabla w}2^2+\frac1{p+1}\,\lambda^{-\,\frac{p-1}{2\,p}\,(d-\gamma)}\,\nrm w{p+1,\gamma}^{p+1}\,.
\]
An optimization on $\lambda>0$ shows that
\[
\min_{\lambda>0}\RF_\gamma[w^\lambda]=\kappa\,\(\nrm w{2p,\gamma}^{2p}\,\mathcal Q_\gamma^{2p}[w]\)^{\theta_\gamma}
\]
for some positive, explicit constant $\kappa$ which continuously depends on $p$, $d$ and $\gamma\in[0,d-(d-2)\,p)$. This proves that
\be{I-C}
\mathsf J_\gamma=\kappa\,\C_\gamma^{-2\,p\,\theta_\gamma}
\ee
is indeed independent of $M$.

It is clear from the above analysis that, up to a multiplication of the function $w$ by a constant, we can fix $M>0$ arbitrarily. We shall furthermore assume that $w_\gamma$ is nonnegative without loss of generality because $|w_\gamma|\in\Hpgamma$, $\RF_\gamma[w_\gamma]=\RF_\gamma[\,|w_\gamma|\,]$ and $\nrm{\,|w_\gamma|\,}{2p,\gamma}=\nrm{w_\gamma}{2p,\gamma}$. By standard arguments, $w_\gamma$ satisfies the semilinear Euler-Lagrange equation
\[
-\,\Delta w_\gamma+|x|^{-\gamma}\(w_\gamma^p-\mu\,w_\gamma^{2p-1}\)=0\quad\mbox{in}\quad\R^d
\]
in distributional sense, for some positive \emph{Lagrange multiplier} $\mu=\mu(M)$. From the scaling properties of $\RF_\gamma$, we find that
\[
\mu(M)=2\,p\,\theta_\gamma\,\mathsf J_\gamma\,M^{\theta_\gamma-1}\,.
\]
Hence we can always take $\mu$ equal to $1$ by choosing
\be{M-I}
M=\(2\,p\,\theta_\gamma\,\mathsf J_\gamma\)^{1/(1-\theta_\gamma)}\,.
\ee
{}From now on, \emph{$w_\gamma$ denotes a solution to~\eqref{prob: min} satisfying the above mass condition and solving the equation}
\be{prob: Eul-Lag-one}
-\,\Delta w_\gamma+|x|^{-\gamma}\(w_\gamma^p-w_\gamma^{2p-1}\)=0\quad\mbox{in}\quad\R^d\,.
\ee
In Section~\ref{sec: bnd-transl}, however, we will use a scaling in order to argue with a different choice of mass for $\gamma=0$.

Using a change of variables and uniqueness results that can be found for instance in~\cite{MR1647924} (also see earlier references therein), we know that the radial \emph{ground state}, that is, any \emph{radial} positive solution converging to $0$ as $|x|\to\infty$, is actually unique (see Lemma~\ref{Lem: final-teo-barenblatt}). Let us define
\[
\eta:=d-\gamma-p\,(d-2)\,,\quad a_\gamma:=\frac{(2-\gamma)\,\eta}{(p-1)^2}\quad\mbox{and}\quad b_\gamma:=\frac{\eta^2}{p\,(p-1)^2}
\]
for any $\gamma\in[0,2)$. The function $w_\star$ defined by~\eqref{Barenblatt-w} solves
\[
-\Delta w_\gamma+\frac{a_\gamma}{|x|^\gamma}\(w_\gamma^p-\frac{a_\gamma}{b_\gamma}\,w_\gamma^{2p-1}\)=0\,.
\]
The rescaled function
\[
w_\gamma^\star(x)=\(\frac{a_\gamma}{b_\gamma}\)^\frac1{p-1}\,w_\star\(b_\gamma^{-\frac1{2-\gamma}}\,x\)
\]
solves~\eqref{prob: Eul-Lag-one} and is explicitly given by
\be{eq: barenblatt-weight-intro}
w_\gamma^\star(x):=\(\frac{a_\gamma}{b_\gamma+|x|^{2-\gamma}}\)^{\frac1{p-1}}\,.
\ee
With these preliminaries in hand, we can state a slightly stronger version of Theorem~\ref{thm: main-thm-barenblatt-gamma-intro}.
%---------------------------------------------------------------------
\begin{thm}\label{thm: main-thm-barenblatt-gamma} Let $d\ge3$. Then for any $p\in(1,d/(d-2))$ there exists $\gamma^\ast=\gamma^\ast(p,d)\in(0,d-(d-2)\,p)$ such that $w_\gamma$ exists and is equal to $w_\gamma^\star$ for all $\gamma\in(0,\gamma^\ast)$.\end{thm}
%---------------------------------------------------------------------
Recall that $w_\gamma$ is a solution to~\eqref{prob: min} such that~\eqref{M-I} and~\eqref{prob: Eul-Lag-one} hold. All other solutions to~\eqref{prob: min} can be deduced using multiplication by constants. We shall prove various intermediate results in Sections~\ref{sect: stime-apriori}--\ref{sec: bnd-transl} and complete the proof of Theorem~\ref{thm: main-thm-barenblatt-gamma} in Section~\ref{sec: radiality}. As mentioned above, it is not restrictive to work with $w_\gamma\ge0$ and we shall therefore consider only nonnegative functions, without further notice. Theorem~\ref{thm: main-thm-barenblatt-gamma} is stronger than Theorem~\ref{thm: main-thm-barenblatt-gamma-intro} because it characterizes all optimal functions and not only the value of the optimal constant. Notice that the existence result of $w_\gamma$ does not require restrictions on $\gamma\in(0,2)$: see Proposition~\ref{eq: teo-exist}.

%%%%%%%%%%%%%%%%%%%%%%%%%%%%%%%%%%%%%%%%%%%%%%%%%%%%%%%%%%%%%%%%%%%%%%
\subsection{Density and compactness results}
%---------------------------------------------------------------------
\begin{lem}\label{prop: density} Let $p$, $d$ satisfy~\eqref{eq: parameters}. Then $\mathcal D(\R^d\setminus\{0\})$ is dense in~$\Hpgamma$.\end{lem}
%---------------------------------------------------------------------
\begin{proof} Let us consider some function $w\in\Hpgamma$. The weight $|x|^{-\gamma}$ is, locally in $\R^d\setminus\{0\}$, bounded and bounded away from zero. By standard mollification arguments, one can pick a sequence of functions $\{w_n\}\subset\mathcal D(\R^d\setminus\{0\})$ such that
\[\label{eq: density-1}
\lim_{n\to\infty}\nrm{w-w_n}{{p+1},\gamma}=0\,,\quad\lim_{n\to\infty}\nrm{\nabla w-\nabla w_n}2=0
\]
if $w$ is compactly supported in $\R^d\setminus\{0\}$. Otherwise a simple truncation shows that it is not restrictive to assume, in addition, that $w\in\mathrm L^\infty(\R^d)$. Let $\xi$ be a smooth function such that
\[
0\le\xi(x)\le1\quad\forall\,x\in\R^d\,,\quad\xi(x)=1\quad\forall\,x\in B_1\,,\quad\xi(x)=0\quad\forall\,x\in B_2^c\,,
\]
where $B_r:=B_r(0)$ and consider
\[\label{eq: density-3}
w_n:=\xi_n\,w\,,\quad\xi_n(x):=\(1-\xi\(n\,x\)\)\,\xi\({x}/{n}\)\quad\forall\,x\in\R^d\,,
\]
for any $n\ge2$. It is clear that $\lim_{n\to\infty}\nrm{w-w_n}{p+1,\gamma}=0$ by dominated convergence. As for $\nabla w_n=w\,\nabla\xi_n+\xi_n\nabla w$, we get that \hbox{$\lim_{n\to\infty}\nrm{\nabla w-\xi_n\nabla w}2=0$} again by dominated convergence so that density is proven as soon as we establish that \hbox{$\lim_{n\to\infty}\nrm{w\,\nabla\xi_n}2=0$}. This follows from
\[
\nrm{w\,\nabla\xi_n}2^2\le|\mathbb S^{d-1}|\,\nrm{\nabla\xi}\infty^2\,\nrm w\infty^2\,n^2\int_{1/n}^{2/n} r^{d-1}\,dr\,+\,\frac{\nrm{\nabla\xi}\infty^2}{n^2}\int_{n\le|x|\le2n}w^2\,dx\,.
\]
Since $d\ge3$, the first term in the r.h.s.~vanishes as $n\to\infty$. As for the second term, we get
\begin{multline*}
\int_{n\le|x|\le2n}w^2\,dx\le(2\,n)^{\frac{2\,\gamma}{p+1}}\int_{n\le|x|\le2n}w^2\,|x|^{-\frac{2\,\gamma}{p+1}}\,dx\\
\le(2\,n)^{\frac{2\,\gamma}{p+1}}\,\nrm w{p+1,\gamma}^2\(|\mathbb S^{d-1}|\int_n^{2n}r^{d-1}\,dr\)^{\frac{p-1}{p+1}}
\end{multline*}
by H\"older's inequality and the r.h.s.~goes to zero as $n\to\infty$ because $\frac{2\,\gamma}{p+1}+d\,\frac{p-1}{p+1}<2$ for any $p<(d-\gamma)/(d-2)$.
\qed
\end{proof}

%---------------------------------------------------------------------
\begin{lem}\label{lem: local-compactness} Let $d\ge3$ and $\gamma\in[0,2)$. Then $\dot{\mathrm H}^1(\R^d)$ is \emph{locally} compactly embedded in $\mathrm L^q_\gamma(\R^d)$ for all $q\in[1,2^\ast_\gamma)$.
\end{lem}
%---------------------------------------------------------------------
\emph{Locally} compactly embedded means that for any bounded sequence $\{w_n\}\subset\dot{\mathrm H}^1(\R^d)$, the sequence $\{\chi\,w_n\}$ is relatively compact in $\mathrm L^q_\gamma(\R^d)$ for any characteristic function $\chi$ of a compact set in $\R^d$.

\begin{proof} As a direct consequence of~\eqref{eq: sobolev-hardy-H} and~\eqref{eq: hardy-H}, we have the equi-integrability estimates
\begin{eqnarray*}\label{Equintegrability1}
\int_{B_r}|w|^q\,|x|^{-\gamma}\,dx&\le&\(\int_{B_r}w^2\,|x|^{-\gamma}\,dx\)^{\frac{2^\ast_\gamma-q}{2^\ast_\gamma-2}}\(\int_{B_r}|w|^{2^\ast_\gamma}\,|x|^{-\gamma}\,dx\)^{\frac{q-2}{2^\ast_\gamma-2}}\\
&\le&\;r^{(2-\gamma)\,\frac{2^\ast_\gamma-q}{2^\ast_\gamma-2}}\,\mathsf C_{\rm H}^{2\,\frac{2^\ast_\gamma-q}{2^\ast_\gamma-2}}\,\mathsf C_{\rm HS}^{2^\ast_\gamma\,\frac{q-2}{2^\ast_\gamma-2}}\,\nrm{\nabla w}2^q \quad\forall\,r>0
\end{eqnarray*}
for any $q\in[2,2^\ast_\gamma)$, and
\begin{eqnarray*}
\int_{B_r}|w|^q\,|x|^{-\gamma}\,dx&\le&\(\int_{B_r}w^2\,|x|^{-\gamma}\,dx\)^\frac q2\(|\mathbb S^{d-1}|\int_0^rs^{d-1-\gamma}\,ds\)^{1-\frac q2}\\
&\le&\;r^{(2-\gamma)\,\frac q2+(d-\gamma)\left(1-\frac q2\right)}\,{\mathsf C_{\rm H}}^q\,\nrm{\nabla w}2^q\(\tfrac1{d-\gamma}\,|\mathbb S^{d-1}|\)^{1-\frac q2}\quad\kern-6pt\forall\,r>0\label{Equintegrability2}
\end{eqnarray*}
for any $q\in[1,2]$. The result follows from the well-known local compactness of subcritical Sobolev embeddings (see, \emph{e.g.},~\cite[Theorem 7.22]{GT}).
\qed
\end{proof}

%---------------------------------------------------------------------
\begin{pro}\label{prop: compactness} Let $p$, $d$ satisfy~\eqref{eq: parameters}. Then $\Hpgamma$ is compactly embedded in $\mathrm L^{2p}_\gamma(\R^d)$.\end{pro}
%---------------------------------------------------------------------
\begin{proof} By definition, $\Hpgamma$ is continuously embedded in $\dot{\mathrm H}^1(\R^d)$ and \emph{locally} compactly embedded in $\mathrm L^{2p}_\gamma(\R^d)$ by Lemma~\ref{lem: local-compactness}. Using H\"older's inequality and then Sobolev's inequality~\eqref{eq: sobolev-H} we get
\begin{multline*}\label{eq: Horiuchi-5}
\int_{B^c_R}\frac{|w|^{2p}}{|x|^\gamma}\,dx\le R^{-\frac\gamma{q^\prime}}\(\int_{B^c_R}\frac{|w|^{p+1}}{|x|^\gamma}\,dx\)^{1/q}\(\int_{B^c_R}|w|^{2^\ast}\,dx\)^{1/{q^\prime}}\\
\le R^{-\frac\gamma{q^\prime}}\,\nrm w{p+1,\gamma}^\frac{p+1}q\,(\mathsf C_{\rm S}\,\nrm{\nabla w}2)^\frac{2^\ast}{q^\prime} \quad\forall\,R>0\,,
\end{multline*}
with $q=\frac{2^\ast-p-1}{2^\ast-2\,p}$ and $q^\prime=\frac q{q-1}$, for any $w\in\Hpgamma$, which is an equi-integrability property at infinity. Then \emph{global} compactness follows.
\qed
\end{proof}

%%%%%%%%%%%%%%%%%%%%%%%%%%%%%%%%%%%%%%%%%%%%%%%%%%%%%%%%%%%%%%%%%%%%%%
\subsection{Existence of optimal functions}
%---------------------------------------------------------------------
\begin{pro}\label{eq: teo-exist} Let $p$, $d$ satisfy~\eqref{eq: parameters}. For any $M>0$, the minimization problem~\eqref{prob: min} admits at least one solution.\end{pro}
%---------------------------------------------------------------------
\begin{proof} The functional $\RF_\gamma$ is coercive and weakly lower semi-continuous on $\Hpgamma$. Since $\Hpgamma$ is a reflexive Banach space, the existence of a solution of~\eqref{prob: min} follows by the compact embedding result of Proposition~\ref{prop: compactness}.
\qed
\end{proof}
%---------------------------------------------------------------------
\begin{oss}\label{oss: norme-senza-peso} Let $p$, $d$ satisfy~\eqref{eq: parameters}. We have $\mathcal H_{p,0}(\R^d)\subset\Hpgamma$ because
\[
\iRd{\frac{|w|^{p+1}}{|x|^\gamma}}\le\(\int_{B_1}\frac{|w|^{2^\ast_\gamma}}{|x|^\gamma}\,dx\)^{\frac{p+1}{2^\ast_\gamma}}\(\int_{B_1}\frac1{|x|^\gamma}\,dx\)^{1-\frac{p+1}{2^\ast_\gamma}}+\int_{B_1^c} |w|^{p+1}\,dx
\]
for all $w\in\mathcal H_{p,0}(\R^d)$.\end{oss}
%---------------------------------------------------------------------

%%%%%%%%%%%%%%%%%%%%%%%%%%%%%%%%%%%%%%%%%%%%%%%%%%%%%%%%%%%%%%%%%%%%%%
%%%%%%%%%%%%%%%%%%%%%%%%%%%%%%%%%%%%%%%%%%%%%%%%%%%%%%%%%%%%%%%%%%%%%%
\section{\emph{A priori} regularity estimates}\label{sect: stime-apriori}

The aim of this section is to provide regularity estimates of the solutions to the Euler-Lagrange equation~\eqref{prob: Eul-Lag-one}. The first result is based on a Moser iterative method, in the spirit of~\cite{MR0159138,MR0288405}. We recall that $2^\ast_\gamma:=2\,\frac{d-\gamma}{d-2}$ and use the notation $2^\ast=2^\ast_0$. To any $q\ge2^\ast$, we associate $\zeta:=\frac{2^\ast}q+\big(1-\frac{2^\ast}q\big)\,\frac{2^\ast_\gamma-2}{2^\ast_\gamma-2\,p}$.
%---------------------------------------------------------------------
\begin{lem}\label{stime-Lq-alte} Let $d\ge3$, $p\in(1,2^\ast\!/2)$ and $\gamma\in\big[0,d-(d-2)\,p\big)$. With the above notations, any solution $w_\gamma$ to~\eqref{prob: Eul-Lag-one} satisfies
\[\label{eq: stima-Lq}
\nrm{w_\gamma}q\le C\,\nrm{\nabla w_\gamma}2^\zeta\quad\forall\,q\in[2^\ast,\infty]
\]
for some positive constant $C$ which depends continuously on $\gamma$, $q$ and $p$ and has a finite limit $C(\infty)$ as $q\to\infty$. \end{lem}
%---------------------------------------------------------------------
\begin{proof} Let us set $\varepsilon_0:=2^\ast_\gamma-2\,p$. For any $A>0$, after multiplying~\eqref{prob: Eul-Lag-one} by the test function $(w_\gamma\wedge A)^{1+\varepsilon_0}$ and integrating by parts in $\R^d$, and then letting $A\to\infty$, we obtain the identity:
\[\label{eq: prima-stima-H1}
\frac{4\,(1+\varepsilon_0)}{(2+\varepsilon_0)^2}\iRd{\left|\nabla{w_\gamma^{1+\varepsilon_0/2}}\right|^2}+\iRd{w_\gamma^{p+1+\varepsilon_0}\,|x|^{-\gamma}}=\iRd{w_\gamma^{2p+\varepsilon_0}\,|x|^{-\gamma}}\,.
\]
By applying the Hardy-Sobolev inequality~\eqref{eq: sobolev-hardy-H} to the function $w=w_\gamma^{1+{\varepsilon_0}/2}$, we deduce that
\[
\nrm{w_\gamma}{2p+\varepsilon_1,\gamma}^{2+\varepsilon_0}\le\frac{(2+\varepsilon_0)^2}{4\,(1+\varepsilon_0)}\,\mathsf C_{\rm HS}^2\,\nrm{w_\gamma}{2p+\varepsilon_0,\gamma}^{2p+\varepsilon_0}
\]
with $2\,p+\varepsilon_1=2^\ast_\gamma\,(1+\varepsilon_0/2)$. Let us define the sequence $\{\varepsilon_n\}$ by the recursion relation $\varepsilon_{n+1}:=2^\ast_\gamma\,(1+\varepsilon_n/2)-2\,p$ for any $n\in\N$, that is,
\[\label{eq: stima-Lq-weighted-rec-solved}
\textstyle\varepsilon_n=\frac{2^\ast_\gamma-2\,p}{2^\ast_\gamma-2}\left[2^\ast_\gamma\big(\frac{2^\ast_\gamma}2\big)^n-2\right]\quad\forall\,n\in\N\,,
\]
and take $q_n=2\,p+\varepsilon_n$. If we repeat the above estimates with $\varepsilon_0$ replaced by $\varepsilon_n$ and $\varepsilon_1$ replaced by $\varepsilon_{n+1}$, we get
\begin{equation}\label{eq: est1}
\nrm{w_\gamma}{2p+\varepsilon_{n+1},\gamma}^{2+\varepsilon_n}\le\frac{(2+\varepsilon_n)^2}{4\,(1+\varepsilon_n)}\,\mathsf C_{\rm HS}^2\,\nrm{w_\gamma}{q_n,\gamma}^{q_n}\,.
\end{equation}
Hence, by iterating~\eqref{eq: est1}, we obtain the estimate
\[
\nrm{w_\gamma}{q_n,\gamma}\le C_n\,\nrm{w_\gamma}{2^\ast_\gamma,\gamma}^{\zeta_n}\quad\mbox{with}\quad{\textstyle\zeta_n=\big(\frac{2^\ast_\gamma}2\big)^n\,\frac{2^\ast_\gamma}{q_n}}
\]
where the sequence $\{C_n\}$ is defined by $C_0=\mathsf C_{\rm{HS}}$ and
\[
C_{n+1}^{2+\varepsilon_n}=\frac{(2+\varepsilon_n)^2}{4\,(1+\varepsilon_n)}\,\mathsf C_{\rm HS}^2\,C_n^{q_n}\quad\forall\,n\in\N\,.
\]
The sequence $\{C_n\}$ converges to a limit $C_\infty$. Letting $n\to\infty$ we get the uniform bound
\[\label{eq: stima-infty-indip}
\nrm{w_\gamma}\infty\le C_\infty\,\nrm{w_\gamma}{2^\ast_\gamma,\gamma}^{\zeta_\infty}
\]
where $\zeta_\infty=\frac{2^\ast_\gamma-2}{2^\ast_\gamma-2\,p}=\lim_{n\to\infty}\zeta_n$. The proof is completed using the H\"older interpolation inequality $\nrm{w_\gamma}q\le\nrm{w_\gamma}{2^\ast}^{2^\ast\!/q}\,\nrm{w_\gamma}\infty^{1-2^\ast\!/q}$,~\eqref{eq: sobolev-hardy-H} and~\eqref{eq: sobolev-H}.
\qed
\end{proof}

%---------------------------------------------------------------------
\begin{lem}\label{lem-rego-calpha} Let $d\ge3$, $p\in(1,2^\ast\!/2)$, $q\in[1,\infty)$ and $\gamma\in\big[0,d-(d-2)\,p\big)$ such that $2\,p\,q\,\gamma<d$. Any solution $w_\gamma$ to~\eqref{prob: Eul-Lag-one} is bounded in $\mathrm W^{2,q}_{\rm loc}(\R^d)$, and locally relatively compact in $C^{1,\alpha}(\R^d)$ with $\alpha=1-d/q$, for any $q>d$. Moreover, there exists a positive constant $C_{\gamma,q,p,d}$ such that
\begin{multline}\label{eq: stima-rego-calpha-bis}
\left\|w_\gamma\right\|_{C^{1,\alpha}\(\overline{B}_1(x_0)\)}\le C_{\gamma,q,p,d}\(\|w_\gamma\|_{\mathrm L^q(B_2(x_0))}\hspace*{-3pt}+\|w_\gamma\|_{\mathrm L^{(2p-1)q}(B_2(x_0))}^{2p-1}\hspace*{-3pt} \right.\\+\left. \|w_\gamma\|_{\mathrm L^{p\kern0.5pt q}(B_2(x_0))}^p+\,\|w_\gamma\|_{\mathrm L^{2p\kern0.5pt q}(B_2(x_0))}^{2p-1}+\|w_\gamma\|_{\mathrm L^{2p\kern0.5pt q}(B_2(x_0))}^p\)
\end{multline}
for any $x_0\in\R^d$, and $\limsup_{\gamma \to 0_+} C_{\gamma,q,p,d}<\infty$.\end{lem}
%---------------------------------------------------------------------
\begin{proof} For any domain $\Omega\in\R^d$, we obtain from the Euler-Lagrange equation~\eqref{prob: Eul-Lag-one} that 
\begin{eqnarray*}
&&2^{1-q}\left\|\Delta w_\gamma\right\|_{\mathrm L^q(\Omega)}^q\\
&&\le\int_\Omega w_\gamma^{(2p-1)\kern0.5pt q}\,dx+\int_\Omega w_\gamma^{p\kern0.5pt q}\,dx+\(\int_{B_1}|x|^{-2\gamma p\kern0.5pt q}\,dx\)^{\frac1{2p}}\(\int_\Omega w_\gamma^{2p\kern0.5pt q}\,dx\)^{\frac{2p-1}{2\,p}}\\
&&\qquad+\(\int_{B_1}|x|^{-2\gamma q}\,dx\)^{\frac12}\(\int_\Omega w_\gamma^{2p\kern0.5pt q}\,dx\)^{\frac12}\\
&&\le2^{1-q}\,C^q\(\left\|w_\gamma\right\|_{\mathrm L^{(2p-1)\kern0.5pt q}(\Omega)}^{(2p-1)\kern0.5pt q}+\left\|w_\gamma\right\|_{\mathrm L^{p\kern0.5pt q}(\Omega)}^{p\kern0.5pt q}+\left\|w_\gamma\right\|_{\mathrm L^{2p\kern0.5pt q}(\Omega)}^{(2p-1)\kern0.5pt q}+\left\|w_\gamma\right\|_{\mathrm L^{2p\kern0.5pt q}(\Omega)}^{p\kern0.5pt q}\)
\end{eqnarray*}
with $2^{1-q}\,C^q=\max\left\{1,\(\int_{B_1}|x|^{-2\gamma p\kern0.5pt q}\,dx\)^{\frac1{2p}},\(\int_{B_1}|x|^{-2\gamma\kern0.5pt q}\,dx\)^{\frac12}\right\}$. As a consequence, for any domain $\Omega\subset\R^d$ and any solution $w_\gamma$ to~\eqref{prob: Eul-Lag-one}, we have
\begin{multline}\label{eq: estimate-compact-laplaciano}
\left\|\Delta w_\gamma\right\|_{\mathrm L^q(\Omega)}\\ \le C\(\left\|w_\gamma\right\|_{\mathrm L^{(2p-1)\kern0.5pt q}(\Omega)}^{2p-1}\hspace*{-3pt}+\left\|w_\gamma\right\|_{\mathrm L^{p\kern0.5pt q}(\Omega)}^p\hspace*{-3pt}+\left\|w_\gamma\right\|_{\mathrm L^{2p\kern0.5pt q}(\Omega)}^{2p-1}\hspace*{-3pt}+\left\|w_\gamma\right\|_{\mathrm L^{2p\kern0.5pt q}(\Omega)}^p\).
\end{multline}

By the Calder\'on-Zygmund theory, we know that there exists a positive constant $C^\prime=C^\prime(q,d)$ such that, for any function $w\in\mathrm W^{2,q}_{\rm loc}(B_2(x_0))\cap\mathrm L^q(B_2(x_0))$ with $\Delta w\in\mathrm L^q(B_2(x_0))$, the inequality
\[\label{eq: cald-zyg-1}
\left\|w\right\|_{\mathrm W^{2,q}(B_1(x_0))}\le C^\prime \(\left\|w\right\|_{\mathrm L^q(B_2(x_0))}+\left\|\Delta w\right\|_{\mathrm L^q(B_2(x_0))}\)
\]
holds. See for instance~\cite[Theorem 9.11]{GT}. \emph{A priori} we do not know whether $w_\gamma\in\mathrm W^{2,q}_{\rm loc}(B_2(x_0))$, but we can consider $w_{\gamma,\varepsilon}:=w_\gamma\ast\rho_\varepsilon$, where $\{\rho_\varepsilon\}$ is a family of mollifiers depending on $\varepsilon>0$, apply the Calder\'on-Zygmund estimate to $w_{\gamma,\varepsilon}$ and then pass to the limit as $\varepsilon\to 0$ with $w_{\gamma,\varepsilon}\to w_\gamma$ and $\Delta w_{\gamma,\varepsilon}\to\Delta w_\gamma$ in $\mathrm L^q(B_2(x_0))$, because of Lemma~\ref{stime-Lq-alte} and of the above estimate on $\left\|\Delta w_\gamma\right\|_{\mathrm L^q(\Omega)}$.

Estimate~\eqref{eq: stima-rego-calpha-bis} and the local relative compactness are then consequences of the standard Sobolev embeddings: see, \emph{e.g.},~\cite[Theorem 7.26]{GT}.
\qed
\end{proof}

\begin{oss}\label{rem: regularity}
Thanks to the uniform bound provided by Lemma~\ref{stime-Lq-alte}, the Euler-Lagrange equation~\eqref{prob: Eul-Lag-one} implies that $|x|^\gamma\Delta w_\gamma\in\mathrm L^\infty(\R^d)$, for all $\gamma$ and $p$ complying with~\eqref{eq: parameters}. Hence, if $\gamma\in(0,1)$ then $w_\gamma\in C^{1,\alpha}$ for all $\alpha<1-\gamma$, while $w_\gamma\in C^{0,\alpha}$ for all $\alpha<2-\gamma$ if $\gamma\in[1,2)$.

The optimal regularity for solutions to~\eqref{prob: Eul-Lag-one} can be estimated by the regularity of the function $w_\gamma^\star$ defined in~\eqref{eq: barenblatt-weight-intro}. This solution is precisely of class $C^{1,1-\gamma}$ for $\gamma\in(0,1)$ and of class $C^{0,2-\gamma}$ for $\gamma\in[1,2)$.
\end{oss}

%%%%%%%%%%%%%%%%%%%%%%%%%%%%%%%%%%%%%%%%%%%%%%%%%%%%%%%%%%%%%%%%%%%%%%
%%%%%%%%%%%%%%%%%%%%%%%%%%%%%%%%%%%%%%%%%%%%%%%%%%%%%%%%%%%%%%%%%%%%%%
\section{Concentration-compactness analysis and consequences}\label{sec: conc-comp}

In this section we shall make use of a suitable variant of the concentration-compactness principle as stated in~\cite{MR778970,MR778974}. We consider the minimization problem $\mathsf J_\gamma$ as defined in Section~\ref{Sec: Scalings and EL} in the limit $\gamma\downarrow 0$. Our goal is to prove that the solutions $w_\gamma$ to problem~\eqref{prob: min} approximate, \emph{up to translations}, the function
\be{eq: expl-minim}
w_0(x):=\(\frac{a_0}{b_0+|x|^2}\)^{\frac1{p-1}}\quad\forall\,x\in\R^d
\ee
defined in~\eqref{eq: barenblatt-weight-intro}. Notice indeed that in the limit case $\gamma=0$ the problem is invariant under translations, which is the major source of difficulties in this section. We will put the emphasis on the differences with the standard results of the concentration-compactness method and refer to~\cite[Section 3.3.1]{MuratoriThesis2015} for fully detailed proofs. Our goal is to establish \emph{a priori} estimates on translations and get a uniform upper bound on the optimal functions as $\gamma\downarrow 0$.
%---------------------------------------------------------------------
\begin{pro}\label{prop: conc-comp} Let $p\in(1,2^\ast\!/2)$. Then we have
\be{eq: conc-comp-4}
\lim_{\gamma\to 0}\mathsf J_\gamma=\mathsf J_0\,.
\ee
Let $\{\gamma_n\}\subset\big(0,d-(d-2)\,p\big)$ be a decreasing sequence such that $\lim_{n\to\infty}\gamma_n=0$. For any $n\in\N$, we consider solutions $w_{\gamma_n}$ to problem~\eqref{prob: min} satisfying~\eqref{prob: Eul-Lag-one} with $\gamma=\gamma_n$. Then
\[\label{eq: conc-comp-1}
\lim_{n\to\infty}\iRd{w_{\gamma_n}^{2p}\,|x|^{-\gamma_n}}=\iRd{w_0^{2p}}=M
\]
and up to the extraction of a subsequence, there exists $\{y_n\}\subset\R^d$ such that
\[
v_n:=w_{\gamma_n}(\cdot+y_n)\rightarrow w_0\quad\mbox{strongly in}\;\dot{\mathrm H}^1(\R^d)\,,
\]
where either $\{y_n\}$ is bounded or $|y_n|\to\infty$ and $\ell:=\lim_{n\to\infty}|y_n|^{\gamma_n}=1$. \end{pro}
%---------------------------------------------------------------------
\begin{proof} According to~\eqref{eq: CKN-originaria-interp} and~\eqref{I-C}, $\mathsf J_\gamma$ is bounded away from $0$ as $\gamma\downarrow 0$. Using $w_0$ as a test function yields
\be{eq: major-I-gamma-2}
\limsup_{\gamma\to0}\mathsf J_\gamma\le\mathsf J_0\,.
\ee
Hence, up to the extraction of a subsequence, $\{\mathsf J_{\gamma_n}\}$ converges to a finite positive limit that we shall denote by $\mathsf J$. According to~\eqref{M-I},
\be{eq-1}
M_n:=\(2\,p\,\theta_{\gamma_n}\,\mathsf J_{\gamma_n}\)^{1/(1-\theta_{\gamma_n})}=\nrm{w_{\gamma_n}}{2p,\gamma_n}^{2p}
\ee
converges as $n\to\infty$ to $M:=\(2\,p\,\theta_0\,\mathsf J\)^{1/(1-\theta_0)}$. Let us define
\[
f_n(x):=w_{\gamma_n}(x)\,|x|^{-\frac{\gamma_n}{2\,p}}
\]
and consider the sequence $\{f_n^{2p}\}$. To prove~\eqref{eq: conc-comp-4}, we have indeed to establish the strong convergence of the sequence in $\mathrm L^1(\R^d)$. Our proof relies on the concentration-compactness method.

A simple estimate based on H\"older's inequality rules out the \emph{concentration} scenario. Consider indeed a function $w\in\dot{\mathrm H}^1(\R^d)$. We get that
\[
\int_{B_r(x_0)}|w|^{2p}\,|x|^{-\gamma}\,dx\le\(\int_{B_r(x_0)}|x|^{-\gamma\kern0.5pt q}\,dx\)^{1/q}\nrm w{2^\ast}^{2p}
\]
with $q=d/\big(d-p\,(d-2)\big)$. By standard symmetrization techniques, the r.h.s.~(with $w=w_{\gamma_n}$) is maximal when $x_0=0$ and therefore bounded by $O(r^{d/q-\gamma})$ using Sobolev's inequality~\eqref{eq: sobolev-H},~\eqref{eq: major-I-gamma-2} and~\eqref{eq-1}, uniformly as $\gamma\downarrow0$. Hence we know that
\be{NoConcentration}
\lim_{r\to 0}\sup_{\begin{array}{c}x_0\in\R^d\\n\in\N\end{array}}\;\int_{B_r(x_0)}f_n^{2p}\,dx=0\,.
\ee

In the remainder of this section, we will need a cut-off function $\xi$ with the following properties: $\xi$ is a smooth function which is supported in~$B_2$, such that $0\le\xi\le1$ and satisfies $\xi\equiv1$ in $B_1$. We shall also use the scaled cut-off function defined by $\xi_r(x)=\xi(x/r)$ for any $x\in\R^d$, $r>0$.

Based on~\cite[Lemma I.1]{MR778970}, only three scenarii remain possible: \emph{vanishing}, \emph{dichotomy}, or \emph{compactness}. Let us consider each of these three cases.

\medskip\noindent\textbf{Vanishing.} \emph{There exists $R>0$ such that}
\[\label{eq: rule-out-van-1}
\lim_{n\to\infty}\sup_{y\in\R^d}\int_{B_R(y)}f_n^{2p}\,dx=0\,.
\]

By definition of $f_n$, we know that $\nrm{f_n}{2p}^{2p}=M_n$ converges to \hbox{$M>0$}. We deduce from~\eqref{NoConcentration} that there exists some $r>0$ such that
\[
\|\xi_r^{1/p}f_n\|_{2p}^{2p}<\frac12\,M
\]
for all $n$. Let $g_n:=(1-\xi_r^2)^\frac1{2p}\,f_n$ and observe that
\[
\nrm{g_n}{2p}^{2p}>\frac12\,M
\]
for $n$ large enough. On other hand one can prove, by means of~\eqref{eq: sobolev-H},~\eqref{eq: hardy-H},~\eqref{eq: major-I-gamma-2} and~\eqref{eq-1}, that $\{\nabla g_n \}$ is bounded in $\mathrm L^2(\R^d)$. In particular,
\[
\lim_{n\to\infty}\sup_{y\in\R^d}\int_{B_R(y)}g_n^{2p}\,dx=0
\]
shows that $\{g_n\}$ converges to $0$ strongly in $\mathrm L^q(\R^d)$ for any $q\in(2\,p,2^\ast)$, according to~\cite[Lemma~I.1]{MR778974}. By means of a three-point H\"older interpolation we can deduce the existence of positive constants $\alpha$, $\beta$, $\sigma$ (with $\alpha+\beta<1$) depending only on $q$, $p$, $d$ such that
\[
\nrm{g_n}{2p}\le r^{-\sigma\,\gamma_n}\,\nrm{w_{\gamma_n}}{p+1,\gamma_n}^\alpha\,\nrm{g_n}q^\beta\,\nrm{w_{\gamma_n}}{2^\ast}^{1-\alpha-\beta}\,.
\]
Letting $n \to \infty$ we get that $\lim_{n\to\infty}\nrm{g_n}{2p}=0$, a contradiction. Vanishing is therefore ruled out.

\medskip\noindent\textbf{Dichotomy.} \emph{There exists $\lambda\in(0,M)$ such that, for every $\epsilon>0$, one can choose $R_0>0$, a monotone sequence $\{R_n\}_{n\ge1}$ with $R_1>4\,R_0$ and $\lim_{n\to\infty}R_n=\infty$, and a sequence $\{y_n\}\subset\R^d$ such that
\be{eq: rule-out-dich-1}
\int_{B_{R_0}(y_n)}f_n^{2p}\,dx\ge\lambda-\epsilon\quad\mbox{and}\quad\int_{B_{R_n}(y_n)}f_n^{2p}\,dx\le\lambda+\epsilon
\ee
for all $n$ large enough.}

We proceed similarly to \cite[Theorem I.2]{MR778974}. Let
\[
\tilde f_n:=\big[\xi_{R_n/4}(\cdot-y_n)\big]^\frac1p\,f_n\,,\quad\tilde g_n:=\big[1-\xi_{R_n/2}^2(\cdot-y_n)\big]^\frac1{2p}\,f_n\,.
\]
By assumption, we know that
\[
\nrm{\tilde f_n}{2p}^{2p}\ge\lambda-\epsilon \quad \textrm{and} \quad \nrm{\tilde g_n}{2p}^{2p}\ge M-\lambda-\epsilon
\]
for all $n$ large enough. By exploiting the left-hand inequality in~\eqref{eq: rule-out-dich-1} one can show that $\{|y_n|^{\gamma_n}\}$ is bounded. Taking advantage of this property and of the fact that $\tilde f_n$ and $\tilde g_n$ have disjoint supports, we can deduce that
\[
\mathsf J_{\gamma_n}\,M^{\theta_{\gamma_n}}=\mathcal G_{\gamma_n}[w_{\gamma_n}]\ge\mathcal G_{\gamma_n}\big[\tilde f_n\,|x|^\frac{\gamma_n}{2\,p}\big]+\mathcal G_{\gamma_n}\big[\tilde g_n\,|x|^\frac{\gamma_n}{2\,p}\big]+O(\epsilon)
\]
as $n\to\infty$, where~$\mathcal G_\gamma$ is the functional of Section \ref{Sec: Scalings and EL} (for detailed computations, see \cite[Proof of Lemma 3.3.1]{MuratoriThesis2015}). Using $\tilde f_n\,|x|^\frac{\gamma_n}{2\,p}$ and $\tilde g_n\,|x|^\frac{\gamma_n}{2\,p}$ as test functions for $\mathsf J_{\gamma_n}$, passing to the limit as $n\to\infty$ and then taking the limit as $\epsilon\to0$, we get that
\[
\mathsf J\,M^{\theta_0}\ge\mathsf J\,\lambda^{\theta_0}+\mathsf J\,(M-\lambda)^{\theta_0}
\]
with $\theta_0:=\frac{d+2-p\,(d-2)}{d-p\,(d-4)}$. Since we know that $\mathsf J$ is positive, this contradicts the assumption that $\lambda\in(0,M)$, so that dichotomy is ruled out as well.

\medskip\noindent\textbf{Compactness.} \emph{The sequence $\{f_n\}$ is relatively compact in $\mathrm L^{2p}(\R^d)$, up to translations.}

Since the vanishing and dichotomy scenarii have been ruled out, under our assumptions there necessarily exists a sequence $\{y_n\}\subset\R^d$ and a function $f\in\mathrm L^{2p}(\R^d)$ such that, up to the extraction of a subsequence,
\[
\lim_{n\to\infty}\nrm{f_n(\cdot+y_n)-f}{2p}=0\,.
\]
We face two cases:\begin{description}
\item[$\bullet$] The sequence $\{y_n\}$ is bounded.
In that case, up to the extraction of a subsequence, $\{w_{\gamma_n}\}$ strongly converges to a limit $w$ in $\mathrm L^{2p}(\R^d)$.
\item[$\bullet$] The sequence $\{y_n\}$ is unbounded and we can assume without restriction that $\lim_{n\to\infty}|y_n|=\infty$.\end{description}
For later purpose, we take $\ell:=1$ in the first case. In the second case, we define
\[\label{eq: compact-definizione-K}
\ell:=\lim_{n\to\infty}|y_n|^{\gamma_n}\in[1,\infty]\,,
\]
up to the extraction of a subsequence. Let us prove that $\ell$ is finite, by contradiction. The compactness means that for any $\varepsilon\in(0,M)$, there is some $R>0$ such that
\begin{multline}\label{eq: e2}
M-\varepsilon\le\int_{B_R(y_n)}f_n^{2p}\,dx=\int_{B_R(0)}w_{\gamma_n}^{2p}(x+y_n)\,|x+y_n|^{-\gamma_n}\,dx\\
\sim|y_n|^{-\gamma_n}\int_{B_R(0)}w_{\gamma_n}^{2p}(x+y_n)\,dx\,.
\end{multline}
Recalling that $\dot{\mathrm H}^1(\R^d)$ is locally compactly embedded in $\mathrm L^{2p}_\gamma(\R^d)$ for any $\gamma\in[0,2)$ by Lemma \ref{lem: local-compactness}, we know that $\lim_{n\to\infty}|y_n|^{-\gamma_n}\int_{B_R(0)}w_{\gamma_n}^{2p}(x+y_n)\,dx=0$ if $\ell=\infty$, which is absurd. This proves that $\ell<\infty$.

Let $v_n:=w_{\gamma_n}(\cdot+y_n)$, and denote by $v$ the weak limit in $\dot{\mathrm H}^1(\R^d)$ of $\{v_n \}$. From~\eqref{eq: e2} we infer that
\[
\frac1{\ell}\int_{\R^d}v^{2p}\,dx=M\,.
\]
By means of weak lower semi-continuity, Fatou's Lemma and~\eqref{eq: major-I-gamma-2}, we obtain
\[
\frac12\int_{\R^d}\left|\nabla v\right|^2\,dx+\frac1{(p+1)\,\ell}\int_{\R^d}v^{p+1}\,dx\le\mathsf J_0\,M^{\theta_0}\,.
\]
Performing the change of variable
\[
w(x):=\lambda^{\frac{d}{2p}}\,\ell^{-\frac1{2p}}\,v(\lambda x)\,,\quad\lambda:=\ell^\frac1{d-p(d-2)}\,,
\]
we deduce that $w$ satisfies $\|w\|_{2p}^{2p}=M$ and
\[
\frac12\int_{\R^d}\left|\nabla w\right|^2\,dx+\frac{\ell^{\frac{(p-1)(d-2)}{2[d-p(d-2)]}}}{p+1}\int_{\R^d}w^{p+1}\,dx\le\mathsf J_0\,M^{\theta_0}\,,
\]
which, if $\ell>1$, is clearly in contradiction with the definition of $\mathsf J_0$.
This proves at once that
\[
\mathsf J=\lim_{n\to\infty}\mathsf J_{\gamma_n}=\mathsf J_0
\]
and
\[
\ell=1\,,
\]
even in the case $\lim_{n\to\infty}|y_n|=\infty$.

Hence, $v$ is optimal for $\mathsf J_0$, so that according to \cite{DD} $v=w_0(\cdot+y)$ for some $y\in\R^d$ and $\{v_n\}$ converges strongly in $\dot{\mathrm H}^1(\R^d)$ to $w_0(\cdot+y)$. Up to the replacement of $y_n$ by $y_n-y$, we may assume with no restriction that $y=0$, which completes the proof.
\qed
\end{proof}

%---------------------------------------------------------------------
\begin{cor}\label{eq: coro-improved-conv} Under the notations and assumptions of Proposition~\ref{prop: conc-comp}, up to the extraction of subsequences, we have
\be{eq: conv-tutti-Lq}
\lim_{n\to\infty}\nrm{v_n-w_0}q=0\quad\forall\,q\in[2^\ast,\infty)
\ee
and
\be{eq: conv-C}
\lim_{n\to\infty}\nrm{v_n-w_0}{C^{1,\alpha}(\R^d)}=0\quad\forall\,\alpha\in(0,1)\,.
\ee
\end{cor}
%---------------------------------------------------------------------
\begin{proof} From Proposition~\ref{prop: conc-comp} and Sobolev's inequality~\eqref{eq: sobolev-H} we know that $\{v_n\}$ converges to $w_0$ in $\mathrm L^{2^\ast}(\R^d)$. By Lemma~\ref{stime-Lq-alte} and Proposition~\ref{prop: conc-comp}, the sequence $\{v_n\}$ is bounded and $w_0$ is also bounded. Identity~\eqref{eq: conv-tutti-Lq} results from H\"older's inequality: $\nrm{v_n-w_0}q\le\nrm{v_n-w_0}{2^\ast}^{2^\ast\!/q}\,\nrm{v_n-w_0}\infty^{1-2^\ast\!/q}$.

In order to prove~\eqref{eq: conv-C}, we shall make use of the two following inequalities:
\be{eq: dis-rozze-holder}
\left\|w\right\|_{C^{1,\alpha}(\R^d)}\le 4\sup_{x_0\in\R^d}\left\|w\right\|_{C^{1,\alpha}\(\overline{B}_1(x_0)\)}
\ee
and, for any $R>0$,
\be{eq: dis-rozze-holder-prova-2}
\sup_{x_0\in B_R}\left\|w\right\|_{C^{1,\alpha}\(\overline{B}_1(x_0)\)}\le 2\,(1+2^{\,\alpha})\sup_{x_0\in B_{R+1}}\left\|w\right\|_{C^{1,\alpha}\(\overline{B}_{1/2}(x_0)\)}
\ee
where $w$ is any function such that the r.h.s.~in~\eqref{eq: dis-rozze-holder} and~\eqref{eq: dis-rozze-holder-prova-2} are finite. Proofs are elementary and left to the reader. Clearly there exists a suitable number $N_{\!R}\in\N$ and a set of points $\{y_k\}_{k=1,\,2,\ldots N_{\!R}}\subset B_{R+1}$ such that, for every $x_0\in B_{R+1}$, $\overline{B}_{1/2}(x_0)\subset\overline{B}_1(y_k)$ for some $k\in\{1,\ldots,N_{\!R}\}$ depending on~$x_0$. Recalling~\eqref{eq: dis-rozze-holder-prova-2}, we get
\[\label{eq: dis-rozze-holder-prova-4}
\sup_{x_0\in B_R}\left\|w\right\|_{C^{1,\alpha}\(\overline{B}_1(x_0)\)}\le 2\,(1+2^{\,\alpha})\max_{k=1,\,2,\ldots N_{\!R}}\left\|w\right\|_{C^{1,\alpha}\(\overline{B}_1(y_k)\)}\,.
\]
Using~\eqref{eq: dis-rozze-holder}, we deduce that
\begin{multline}\label{eq: dis-rozze-holder-finale}
\hspace*{-6pt}\left\|w\right\|_{C^{1,\alpha}(\R^d)}\\
\le\max\Big\{8\,(1+2^{\,\alpha})\max_{k=1,\,2,\ldots N_{\!R}}\left\|w\right\|_{C^{1,\alpha}\(\overline{B}_1(y_k)\)},4\,\sup_{x_0\in B_R^c}\left\|w\right\|_{C^{1,\alpha}\(\overline{B}_1(x_0)\)}\Big\}\,.
\end{multline}

Given $\alpha\in(0,1)$, let $q=d/(1-\alpha)$. In view of Lemma~\ref{lem-rego-calpha}, there exists a positive constant $C$ depending only on $\alpha$, $p$ and $d$ such that, for $n$ large enough, for any $x_0\in\R^d$,
\begin{multline}\label{eq: stima-controllo-integrale-x0}
\left\|v_n\right\|_{C^{1,\alpha}\(\overline{B}_1(x_0)\)}\le C\left(\|v_n\|_{\mathrm L^q(B_2(x_0))}\hspace*{-3pt}+\|v_n\|_{\mathrm L^{(2p-1)q}(B_2(x_0))}^{2p-1}\hspace*{-3pt}+\|v_n\|_{\mathrm L^{p\kern0.5pt q}(B_2(x_0))}^p\right.\\
\left.+\,\|v_n\|_{\mathrm L^{2p\kern0.5pt q}(B_2(x_0))}^{2p-1}+\|v_n\|_{\mathrm L^{2p\kern0.5pt q}(B_2(x_0))}^p\right)\,.
\end{multline}
Thanks to~\eqref{eq: conv-tutti-Lq} and~\eqref{eq: stima-controllo-integrale-x0}, for every $\varepsilon>0$ there exist $R>0$ and $n_0\in\N$ such that
\[\label{eq: stima-controllo-holder-vn-u}
\sup_{x_0\in B^c_R}\left\|v_n-w_0\right\|_{C^{1,\alpha}\(\overline{B}_1(x_0)\)}\le\varepsilon\quad\forall\,n\ge n_0\,.
\]
In case $q\in(d,2^\ast)$ and $d=3$, one more H\"older interpolation is needed. Using~\eqref{eq: dis-rozze-holder-finale} with $w=v_n-w_0$, we obtain:
\[
\left\|v_n-w_0\right\|_{C^{1,\alpha}(\R^d)}\le\max\Big\{8\,(1+2^{\,\alpha})\,\max_{k=1,\,2,\ldots N_{\!R}}\left\|v_n-w_0\right\|_{C^{1,\alpha}\(\overline{B}_1(y_k)\)},\,4\,\varepsilon\Big\}
\]
for all $n\ge n_0$. From~\eqref{eq: stima-controllo-integrale-x0}, we know that $\{v_n\}$ is bounded in $C^{1,\alpha}\(\overline{B}_1(x_0)\)$ for all $x_0\in\R^d$, relatively compact by Lemma~\ref{lem-rego-calpha} and, as a consequence,
\[\label{eq: conve-locale-holder}
\lim_{n\to\infty} \left\|v_n-w_0\right\|_{C^{1,\alpha}\(\overline B_1(y_k)\)}=0\quad\forall\,k=1,\,2\ldots N_{\!R}\,.
\]
This concludes the proof.
\qed
\end{proof}

A uniform upper bound on $v_n$ results from Proposition~\ref{prop: conc-comp} and Corollary~\ref{eq: coro-improved-conv}. The proof relies on the use of a barrier function and of the Maximum Principle.
%---------------------------------------------------------------------
\begin{pro}\label{lem: barriera} Under the notations and assumptions of Proposition~\ref{prop: conc-comp}, there exists a positive constant $C$ and a positive integer $N$ such that
\[\label{eq: stima-barriera-globale}
v_n(x)\le C\,\big(1+|x|\big)^{-\frac{2-\gamma_n}{p-1}}\quad\forall\,x\in\R^d\,,\quad\forall\,n\ge N\,.
\]
\end{pro}
%---------------------------------------------------------------------
\begin{proof} In view of~\eqref{prob: Eul-Lag-one}, $v_n$ solves
\[\label{eq: equazione-v_n}
-\,\Delta v_n=\frac{v_n^{2p-1}-v_n^p}{|x+y_n|^{\gamma_n}}\quad\mbox{in}\quad\R^d\,.
\]
In view of~\eqref{eq: conv-C}, and recalling the explicit profile of $w_0$ given by~\eqref{eq: expl-minim}, we infer that there exist $R_0>0$ and $n_0\in\N$ such that
\[\label{eq: equazione-v_n-1}
v_n(x)\le 2^{-\frac1{p-1}}\quad\forall\,x\in B^c_{R_0}\,,\quad\forall\,n\ge n_0\,.
\]
In particular,
\[\label{eq: equazione-v_n-sottosol-1}
-\,\Delta v_n\le-\,\frac{v_n^p}{2\,|x+y_n|^{\gamma_n}}\quad\mbox{in}\quad B^c_{R_0}\,,\quad\forall\,n\ge n_0\,.
\]
Using the fact that $=\lim_{n\to\infty}|y_n|^{\gamma_n}=\ell=1$, an elementary computation allows us to prove that there exist $R_1>0$ and $n_1\in\N$ such that
\[\label{eq: dis-elementary}
\left|x+y_n\right|^{\gamma_n}\le 2\,|x|^{\gamma_n}\quad\forall\,x\in B^c_{R_1}\,,\quad\forall\,n\ge n_1\,.
\]
Let $R_2:=\max\{R_0,\,R_1\}$ and $n_2:=\max\{n_0,\,n_1\}$. We infer that $v_n$ satisfies
\be{eq: equazione-v_n-sottosol-2}
-\,\Delta v_n\le-\,\frac{v_n^p}{4\,|x|^{\gamma_n}}\quad\mbox{in}\quad B^c_{R_2}\,,\quad\forall\,n\ge n_2\,.
\ee
The function
\[
\widehat v_n(x):=C_n\,|x|^{-\frac{2-\gamma_n}{p-1}}\quad\mbox{with}\quad{\textstyle C_n^{p-1}\ge4\,\frac{2-\gamma_n}{p-1}\big(\frac{2-\gamma_n}{p-1}+2-d\big)}
\]
is a supersolution to~\eqref{eq: equazione-v_n-sottosol-2}, where $C_n>0$ can be chosen to be bounded independently of $n$ and such that
\[\label{eq: condizione-soprasol-bordo}
\widehat{v}_n(x)\ge v_n(x)\quad\forall\,x\in\partial B_{R_2}\,,\quad\forall\,n\ge n_3\,,
\]
for some $n_3$ large enough. This can be done because, from~\eqref{eq: conv-C}, we know that $\{v_n\}$ is bounded uniformly by a constant independent of $n$, for $n$ large enough. By applying the Maximum Principle, we then obtain that
\[\label{eq: last-barriera}
v_n(x)\le\widehat{v}_n(x)\quad\forall\,x\in B_{R_2}^c\,,\quad\forall\,n\ge N:=\max\{n_2,\,n_3\}\,.
\]
This concludes the proof.
\qed
\end{proof}

%%%%%%%%%%%%%%%%%%%%%%%%%%%%%%%%%%%%%%%%%%%%%%%%%%%%%%%%%%%%%%%%%%%%%%
%%%%%%%%%%%%%%%%%%%%%%%%%%%%%%%%%%%%%%%%%%%%%%%%%%%%%%%%%%%%%%%%%%%%%%
\section{Analysis of the asymptotic translation invariance}\label{sec: bnd-transl}

Proposition~\ref{prop: conc-comp} establishes the convergence of $\{v_n\}=\{w_{\gamma_n}(\cdot+y_n)\}$ to $w_0$ for some sequence of translations $\{y_n\}$. Proceeding in the spirit of~\cite{AB}, we prove that $\{y_n\}$ is necessarily bounded, which directly entails the convergence of $\{w_{\gamma_n}\}$ to $w_0(\cdot-\overline y)$ for some $\overline y\in\R^d$, up to the extraction of a subsequence. Finally, using a \emph{Selection Principle}, we shall prove that $\overline y=0$.
%---------------------------------------------------------------------
\begin{lem}\label{lem: bound-trans1} With the notations of Proposition~\ref{prop: conc-comp}, we have
\[\label{eq: bound-deriv1}
\limsup_{n\to\infty}\iRd{\(\kappa_n\,|x|^{-\gamma_n}\, w_{\gamma_n}^{2p}-\,\tfrac1{p+1}\,w_{\gamma_n}^{p+1}\)\log|x|}<\infty\, ,
\]
where $\{\kappa_n \}$ is a suitable sequence converging to $\tfrac1{2p}$.
\end{lem}
%---------------------------------------------------------------------
The key point of the proof basically relies on an estimate on the derivative with respect to $\gamma$ of the function $\gamma\mapsto\mathcal E_\gamma[w_\gamma]$ at $\gamma=0$. This cannot be done directly because of the unknown value of $\mathsf J_\gamma$, but the difficulty is overcome by adjusting the mass.

\begin{proof} Recall that $\mathcal H_{p,0}(\R^d)\subset\mathcal H_{p,\gamma}(\R^d)$ for any $\gamma\in(0,2)$ such that~\eqref{eq: parameters} holds, according to Remark~\ref{oss: norme-senza-peso}; moreover, in view of Proposition \ref{lem: barriera}, we have that $w_{\gamma_n}\in\mathcal H_{p,0}(\R^d)$ for all $n$ large enough. For any $w\in\mathcal H_{p,0}(\R^d)$, we can therefore write
\[
\mathcal E_\gamma[w]-\mathcal E_0[w]=\gamma\,\mathsf D_\gamma[w]+\(\mathsf J_0-\mathsf J_\gamma\)\nrm w{2p}^{2\kern0.5pt p\kern0.5pt\theta_0}
\]
with
\[\label{eq: bound-deriv4}
\mathsf D_\gamma[w]:=\tfrac1{p+1}\iRd{w^{p+1}\,\frac{|x|^{-\gamma}-1}\gamma}+\mathsf J_\gamma\frac{\nrm w{2p}^{2\kern0.5pt p\kern0.5pt\theta_0}-\nrm w{2p,\gamma}^{2\kern0.5pt p\kern0.5pt\theta_\gamma}}\gamma\,.
\]

Let us introduce the rescaled profile $W_n$ defined by
\[
W_n(x)=\beta_n^\frac2{p-1}\,w_0(\beta_n\,x)\quad\forall\,x\in\R^d\,,
\]
where the scale $\beta_n:=(m_n\,/M)^{\frac{p-1}{d-p\,(d-4)}}$ is such that
\[
\nrm{W_n}{2p}^{2p}=\nrm{w_{\gamma_n}}{2p}^{2p}=:m_n\,,
\]
and recall that $W_n$ is the unique radial minimizer of $\mathcal E_0$ with mass $m_n$. Since $w_{\gamma_n}$ (resp.~$W_n$) minimizes $\mathcal E_{\gamma_n}$ over $\mathcal H_{p,\gamma_n}(\R^d)$ (resp.~$\mathcal E_0$ over $\mathcal H_{p,0}(\R^d)$), as sketched in the introduction, we get that
\[\label{eq: bound-deriv2}
\mathcal E_{\gamma_n}[w_{\gamma_n}]-\mathcal E_0[w_{\gamma_n}]\le\mathcal E_{\gamma_n}[W_n]-\mathcal E_0[W_n]\,,
\]
that is, after dividing by $\gamma_n$,
\be{eq: bound-deriv3}
\mathsf D_{\gamma_n}[w_{\gamma_n}] \le \mathsf D_{\gamma_n}[W_n]
\ee
because the terms involving $(\mathsf J_0-\mathsf J_{\gamma_n})$ cancel out thanks to the particular choice of profile $W_n$.

We recall that $M_n=\(2\,p\,\theta_{\gamma_n}\,\mathsf J_{\gamma_n}\)^{1/(1-\theta_{\gamma_n})}=\nrm{w_{\gamma_n}}{2p,\gamma_n}^{2p}$ according to~\eqref{M-I}. Performing a first order expansion, we get
\be{eq: bound-deriv5}
m_n^{\theta_0}-M_n^{\theta_{\gamma_n}}=\theta_0\,\mu_n^{\theta_0-1}\iRd{w_{\gamma_n}^{2p}\(1-|x|^{-\gamma_n}\)}\\
+M_n^{\vartheta_n}\,\log M_n\,(\theta_0-\theta_{\gamma_n})
\ee
for some intermediate values $\mu_n\in(M_n,m_n)$ and $\vartheta_n\in(\theta_0,\theta_{\gamma_n})$. A similar identity holds for $\|W_n\|_{2p}^{2p\,\theta_0}-\|W_n\|_{2p,\gamma_n}^{2p\,\theta_{\gamma_n}}$.

Thanks to Propositions~\ref{prop: conc-comp} and \ref{lem: barriera}, we deduce that $w_{\gamma_n}(\cdot+y_n)$ converges to $w_0$ in $\mathrm L^{2p}(\R^d)$, so that
\[
\lim_{n\to\infty}m_n=\lim_{n\to\infty}M_n=\lim_{n\to\infty}\mu_n=M\,.
\]
Since $\theta_{\gamma_n}\to\theta_0$, $(\theta_0-\theta_{\gamma_n})/\gamma_n\to-\,\theta'_0=\tfrac{(d-2)\,(p-1)^2}{(d-p\,(d-4))^2}$ and, according to~\eqref{eq: conc-comp-4}, $\mathsf J_{\gamma_n}\to\mathsf J_0$ as $n\to\infty$, as a consequence of~\eqref{eq: bound-deriv3} and~\eqref{eq: bound-deriv5} there holds
\begin{multline*}
\lim_{n\to\infty}\iRd{\(\tfrac1{p+1}\,w_{\gamma_n}^{p+1}\!-\! \kappa_n \,w_{\gamma_n}^{2p}\) \tfrac{|x|^{-\gamma_n}-1}{\gamma_n}}+\tfrac1{2p}\,M \log M\,\tfrac{\theta_0'}{\theta_0}\\
\le \lim_{n\to\infty}\iRd{\(\tfrac1{p+1}\,W_{n}^{p+1}\!-\! \tfrac1{2p} \,W_{n}^{2p}\)\tfrac{|x|^{-\gamma_n}-1}{\gamma_n}}+\tfrac1{2p}\,M \log M\,\tfrac{\theta_0'}{\theta_0}\,,
\end{multline*}
with $\kappa_n := \theta_0\,\mathsf J_{\gamma_n}\, \mu_n^{\theta_0-1} \to \tfrac1{2p}$.

Using the elementary convexity estimates
\[\label{eq: ineq-log-gamma}
1-|x|^{-\gamma}\le\gamma\log|x|\quad\mbox{and}\quad 1-|x|^{-\gamma}\ge\gamma\,|x|^{-\gamma}\log|x|\quad\forall\,x\in\R^d\setminus\{0\}\,,
\]
and the fact that $\beta_n \to 1$, we conclude the proof.
\qed
\end{proof}

%---------------------------------------------------------------------
\begin{cor}\label{coro: convergenza-u-gamma} With the notations of Proposition~\ref{prop: conc-comp}, there exists $\overline y\in\R^d$ such that, up to the extraction of a subsequence,
\[
\lim_{n\to\infty}\left\|w_{\gamma_n}-w_0(\cdot-\overline y)\right\|_q=\lim_{n\to\infty}\left\|w_{\gamma_n}-w_0(\cdot-\overline y)\right\|_{C^{1,\alpha}(\R^d)}=0
\]
for any $q\in\(d\,\frac{p-1}2 ,\infty\)$ and $\alpha\in(0,1)$. Moreover there exists $N\in\N$ and $C>0$ such that
\be{eq: stima-barriera-globale-u-gamma}
w_{\gamma_n}(x)\le C\(1+|x|\)^{-\frac{2-\gamma_n}{p-1}}\quad\forall\,x\in\R^d\,,\quad\forall\,n\ge N\,.
\ee
\end{cor}
%---------------------------------------------------------------------
\begin{proof} Let us prove that $\{y_n\}$ is bounded. Assume by contradiction that $|y_n|\to+\infty$. With $v_n=w_{\gamma_n}(\cdot+y_n)$ we obtain that
\be{e5}
\begin{aligned}
&\iRd{\(\kappa_n\,|x|^{-\gamma_n}\,w_{\gamma_n}^{2p}-\,\tfrac{w_{\gamma_n}^{p+1}}{p+1}\)\log|x|}\\
=&\iRd{\(\kappa_n\,|x+y_n|^{-\gamma_n}\,v_n^{2p} -\,\tfrac{v_n^{p+1}}{p+1}\)\log|x+y_n|}\,.
\end{aligned}
\ee
By means of Propositions \ref{prop: conc-comp}, \ref{lem: barriera} and of the fact that $\kappa_n \to \tfrac1{2p}$, long but elementary computations show that the r.h.s.~of~\eqref{e5} behaves, as $n \to \infty$, like
\[
\log|y_n|\,\iRd{\(\tfrac{w_0^{2p}}{2\,p}-\,\tfrac{w_0^{p+1}}{p+1}\)}\,.
\]
For details we refer again to \cite[Proof of Lemma 3.3.1]{MuratoriThesis2015}. On the other hand, with $\gamma=0$, using~\eqref{M-I} and an identity obtained by multiplying~\eqref{prob: Eul-Lag-one} by $w_0$ and integrating over $\R^d$, we get
\be{eq: i-f}
\iRd{\(\tfrac{w_0^{2p}}{2\,p}-\,\tfrac{w_0^{p+1}}{p+1}\)}=\tfrac{(p-1)\,(d-2)\,M}{2\,p\,\(d+2-p\,(d-2)\)}>0\,.
\ee
This contradicts Lemma~\ref{lem: bound-trans1}.

Hence we can extract a subsequence such that $\{y_n\}$ converges to $\overline y$ and get the convergence result by applying Proposition~\ref{prop: conc-comp} and Corollary~\ref{eq: coro-improved-conv} if $q\ge2^\ast$. The uniform estimate~\eqref{eq: stima-barriera-globale-u-gamma} directly follows from Proposition~\ref{lem: barriera}. To cover the range $q\in(d\,(p-1)/2,2^\ast)$, we observe that the r.h.s.~of such an estimate belongs to $\mathrm L^q(\R^d)$ for any $n$ large enough.
\qed
\end{proof}

%---------------------------------------------------------------------
\begin{pro}\label{lem: minimizer} With the notations of Corollary~\ref{coro: convergenza-u-gamma}, we have $\overline y=0$.
\end{pro}
%---------------------------------------------------------------------
In other words, we prove that $\{w_\gamma\}$ converges to $w_0$ as $\gamma\downarrow0$. This means that, among all the solutions of problem~\eqref{prob: min} at $\gamma=0$, the sequence $\{w_{\gamma_n}\}$ \emph{selects the one centered at zero}. We shall proceed by means of a \emph{Selection Principle} argument, inspired again by~\cite{AB}.

\begin{proof} Let us define
\[
F(y):=\iRd{\(\tfrac{w_0^{2p}}{2\,p}-\tfrac{w_0^{p+1}}{p+1}\)\log|x+y|}\quad\forall\,y\in\R^d\,.
\]
We proceed as in the proof of Lemma~\ref{lem: bound-trans1} but replace $W_n$ with $W_n(\cdot-y)$, for any arbitrary $y\in\R^d$, and obtain that
\[
\mathsf D_{\gamma_n}[w_{\gamma_n}] \le \mathsf D_{\gamma_n}[W_n(\cdot-y)]\,.
\]
By passing to the limit as $n\to\infty$, which is feasible thanks to Corollary \ref{coro: convergenza-u-gamma}, we get that $F(\overline y)\le F(y)$. This proves that
\[
\overline y\in\underset{y\in\R^d}{\operatorname{argmin}}\,F\,.
\]

Next we may consider the function $\mathcal K$ such that
\[
\mathcal K(r)=\tfrac{w_0^{2p}(x)}{2\,p}-\tfrac{w_0^{p+1}(x)}{p+1}\,,
\]
with $r=|x|$ for any $x\in\R^d$. A computation based on the explicit profile~\eqref{eq: expl-minim} of $w_0$, together with~\eqref{eq: i-f}, shows that there exists $R>0$ such that
\[\label{eq: segno-H(r)}
\mathcal K(r)\ge0\quad\forall\,r\in[0,R]\,,\quad\mathcal K(r)< 0\quad\forall\,r>R\,,\quad\mbox{and}\quad\int_0^\infty\mathcal K(r)\,r^{d-1}\,dr>0\,.
\]
Let us choose $\mathsf e\in\mathbb S^{d-1}$, consider the angle $\theta\in[0,\pi]$ such that $\mathsf e\cdot\frac xr=\cos\theta$ and define the function $G\in C^1(\R^+)$ by
\[
G(t):=\iRd{\mathcal K(|x|)\,\log\big|\,x+\sqrt t\,\mathsf e\,\big|}\,.
\]
An elementary computation yields
\[
G'(t)=\frac{|\mathbb S^{d-2}|}t\int_0^\infty\mathcal K(r)\,\ell\(\frac{r^2}t\)\,r^{d-1}\,dr
\]
with
\[
\ell(s):=\int_0^{\pi/2}\frac{1-s\,\cos(2\theta)}{1+s^2-2\,s\,\cos(2\theta)}\,(\sin\theta)^{d-2}\,d\theta\quad\forall\,s\in\R^+\,.
\]
The function $\ell$ is continuous, monotone decreasing as we shall see next, and \hbox{$\lim_{s\to\infty}\ell(s)=0$}. As a consequence, we obtain that $\ell(s)>0$ for any $s\in\R^+$ and
\[
G'(t)\ge\frac{|\mathbb S^{d-2}|}t\,\ell\(\frac{R^2}t\)\int_0^\infty\mathcal K(r)\,r^{d-1}\,dr>0\,,
\]
which shows that the minimum of $G(t)$ is attained at $t=0$ and nowhere else, which is equivalent to proving the statement. To complete the proof, let us give some details concerning the above properties of the function $\ell$.

The continuity of $\ell$ is straightforward since a Taylor expansion around $(\theta,s)=(0,1)$ shows that
\[
\frac{1-s\,\cos(2\theta)}{1+s^2-2\,s\,\cos(2\theta)}\,(\sin\theta)^{d-2}\sim\frac{2\,\theta^2+1-s}{4\,\theta^2+(1-s)^2}\theta^{d-2}+\theta^{d-2}\,O\big(\theta^2+1-s\big)\,.
\]
To prove that $\ell'(s)<0$, we first note that
\begin{multline*}
\ell(s)=\frac12\,\int_0^{\pi/2}(\sin\theta)^{d-2}\,d\theta+\frac 12\,m_d(s)\\
\mbox{where}\quad m_d(s):=\int_0^{\pi/2}\frac{1-s^2}{(1+s)^2-\,4\,s\,(\cos\theta)^2}\,(\sin\theta)^{d-2}\,d\theta\,.
\end{multline*}
We then observe that:
\begin{enumerate}
\item For any $s>0$, $m_d(s)=-\,m_d(1/s)$, so that it is enough to prove that $m_d'(s)<0$ for any $s\in(0,1)$.
\item As a function of $d\ge3$ and for any given value of $s\in(0,1)$, it turns out that $d\mapsto m_d'(s)$ is non increasing. Hence, it is enough to prove that $m_3^\prime(s)<0$.
\item We can explicitly compute
\[
m_3(s)=\frac{1-s}{2\,\sqrt s}\,\operatorname{arctanh}\(\frac{2\,\sqrt s}{1+s}\)\quad\forall\,s\in(0,1)
\]
and check, by means of Taylor expansions, that in fact $m_3^\prime(s)<0$.
\end{enumerate}
This concludes the proof. More details can be found in~\cite[Proof of Lemma 3.3.8]{MuratoriThesis2015}.
\qed
\end{proof}

%%%%%%%%%%%%%%%%%%%%%%%%%%%%%%%%%%%%%%%%%%%%%%%%%%%%%%%%%%%%%%%%%%%%%%
%%%%%%%%%%%%%%%%%%%%%%%%%%%%%%%%%%%%%%%%%%%%%%%%%%%%%%%%%%%%%%%%%%%%%%
\section{Optimal functions are radial for \texorpdfstring{$\gamma$}{gamma} small}\label{sec: radiality}

This section is devoted to the proof of Theorem~\ref{thm: main-thm-barenblatt-gamma}. We start with two technical results.

\medskip First of all, the optimal functions for~\eqref{CKN} among \emph{radially symmetric} functions, that is, functions in $\mathcal H^{\star}_{p,\gamma}(\R^d)$, are based on Barenblatt-type profiles.
%---------------------------------------------------------------------
\begin{lem}\label{Lem: final-teo-barenblatt} Let $p$ and $d$ satisfy~\eqref{eq: parameters}. Then the solution to problem~\eqref{prob: min} restricted to $\mathcal H^{\star}_{p,\gamma}(\R^d)$ is unique and explicit. If the mass $M$ is chosen so that it satisfies the Euler-Lagrange equation~\eqref{prob: Eul-Lag-one}, it coincides with $w_\gamma^\star$ as in~\eqref{eq: barenblatt-weight-intro}.\end{lem}
%---------------------------------------------------------------------
\begin{proof} Since every solution to problem~\eqref{prob: min} restricted to $\mathcal H^{\star}_{p,\gamma}(\R^d)$ is a solution to~\eqref{prob: Eul-Lag-one} after a suitable rescaling, it is enough to establish uniqueness for nonnegative, nontrivial solutions to~\eqref{prob: Eul-Lag-one} belonging to $\mathcal H^{\star}_{p,\gamma}(\R^d)$. Hereafter we shall denote any such solution as $w_\gamma^\star$ and, with a slight abuse of notation, write $w_\gamma^\star(|x|)=w_\gamma^\star(x)$ for any $x\in\R^d$. If we perform the change of variables
\[\label{eq: ch-var}
v(s)=w_\gamma^\star\big(c\,s^{2/(2-\gamma)}\big)\quad\forall\,s\in\R^+
\]
where $c:=\left[(2-\gamma)/2\right]^{2/{(2-\gamma)}}$, then $v$ solves
\[\label{eq: ch-var-diff-eq}
-v^{\prime\prime}-\frac{d_\gamma-1}{s}\,v^{\prime}+v^p=v^{2p-1}\quad\mbox{in}\quad\R^+\,,
\]
where $d_\gamma:=2\,(d-\gamma)/(2-\gamma)$. From the $\mathrm L^\infty$ bound found in Lemma~\ref{stime-Lq-alte} and from the Calder\'on-Zygmund theory (see the proof of Lemma~\ref{lem-rego-calpha}), we easily get that $w_\gamma^\star\in\mathrm L^\infty(\R^d)\cap C^1(\R^d\setminus\{0\})$. The function $v$ is actually of class~$C^1$: see~\cite[Section 6, Remark 3]{MR1647924} for a proof. Uniqueness then follows from~\cite[Theorem~2]{MR1647924}, so that $w_\gamma^\star$ does coincide with the Barenblatt-type profile defined by~\eqref{eq: barenblatt-weight-intro}.

We may notice that the above change of variables amounts to rewrite the radial problem in a ``dimension'' $d_\gamma$, which is not necessarily an integer, without weight, and therefore reduces the uniqueness issue to a problem that has already been considered in~\cite{DD}.
\qed
\end{proof}

The second lemma is a spectral gap property, which is a consequence of various results that can be found in~\cite{DenzMc-pnas,DenzMc05,BBDGV-CRAS,BBDGV,BDGV-pnas}. We recall that, according to the conventions of the introduction and the results of \cite{DD}, we have that $w_0=w_0^\star$.
%---------------------------------------------------------------------
\begin{lem}\label{lem: hardy-poinc} For any function $\omega\in\dot{\mathrm H}^1(\R^d)$ such that
\be{eq: func-HP-zeromean}
\iRd{w_0^{2p-1}\,\omega}=0\,,
\ee
the inequality
\be{eq: func-HP}
\iRd{|\nabla\omega|^2}+p\iRd{w_0^{p-1}\,\omega^2}\ge(2\,p-1)\iRd{w_0^{2(p-1)}\,\omega^2}
\ee
holds with equality if and only if $\omega=a\cdot\nabla w_0$ for some $a\in\R^d$.\end{lem}
%---------------------------------------------------------------------
\begin{proof} A second order Taylor expansion of $\mathcal E_0[w_0+\varepsilon\,\omega]$ defined by~\eqref{eq: E-funz-energy} in terms of $\varepsilon$ and the fact that $w_0$ is a minimizer of~$\mathcal E_0$ establishes the inequality.

To identify the equality case, we introduce $f$ such that $\omega=f\,w_0^p$. Since $\nabla\omega=w_0^p\,\nabla f+f\,\nabla(w_0^p)$, we obtain after integrating by parts that
\begin{multline*}
\iRd{|\nabla\omega|^2}-\iRd{|\nabla f|^2\,w_0^{2p}}=-\iRd{f^2\,w_0^p\,\Delta(w_0^p)}\\
=-\,p\iRd{f^2\,w_0^{2p-1}\,\(\Delta w_0+(p-1)\,\frac{|\nabla w_0|^2}{w_0}\)}\,.
\end{multline*}
Since $w_0$ solves $-\,\Delta w_0+w_0^p-w_0^{2p-1}=0$ and, using the notations of Section~\ref{Sec: Scalings and EL},
\[
\frac{|\nabla w_0|^2}{w_0}=\frac4{(p-1)^2}\(\frac1{a_0}\,w_0^p-\frac{b_0}{a_0^2}\,w_0^{2p-1}\),
\]
inequality~\eqref{eq: func-HP} can be rewritten in terms of $f$ as the Hardy-Poincar\'e inequality
\be{eq: e5}
\iRd{|\nabla f|^2\,w_0^{2p}}\ge\tfrac{2\,p\,(p-1)}{d-p\,(d-2)}\iRd{|f|^2\,w_0^{3p-1}}\,,
\ee
while condition~\eqref{eq: func-HP-zeromean} amounts to $\iRd{f\,w_0^{3p-1}}=0$. According to~\cite[see pp.~16462--16463]{BDGV-pnas} (we also refer to~\cite{DenzMc-pnas,DenzMc05} for earlier results, but in a different functional setting), the equality case corresponds to $f(x)=x\cdot a$ for some $a\in\R^d$. We point out that such an $f$ belongs to the closure of $\mathcal{D}(\R^d)$ with respect to the (square) norm identified by the l.h.s.~of~\eqref{eq: e5}. The proof is completed by noting that $\nabla w_0$ is proportional to $w_0^p\,x$.
\qed
\end{proof}

\noindent{\bf\emph{Proof of Theorem~\ref{thm: main-thm-barenblatt-gamma}.}}

We argue by contradiction using the \emph{angular derivatives} of possibly non-radial optimal functions. Given a nontrivial antisymmetric matrix $\mathsf A$ and a differentiable function $f$, we define the angular derivative of $f$ with respect to~$\mathsf A$ by
\[\label{eq: def-deriv-ang}
\nabla_{\kern-2pt\mathsf A}f(x):= \mathsf A\,x \cdot \nabla f(x)=\frac d{dt}f\(e^{t\kern0.5pt\mathsf A}\,x\)_{|t=0}\quad\forall\,x\in\R^d\,.
\]
It turns out that a function $f$ is radial if and only if $\nabla_{\kern-2pt\mathsf A}f \equiv 0$ for any antisymmetric matrix $\mathsf A$. Assume by contradiction that $ w_{\gamma_n}$ is non radial for some sequence $\{\gamma_n\}$ with $\gamma_n\downarrow0$, \emph{i.e.},~there exists an antisymmetric matrix $\mathsf A_n$ such that
\[\label{eq: deriv-ang-normalized}
\omega_n:=\nabla_{\kern-2pt\mathsf A_n}w_{\gamma_n}\not\equiv0\quad\mbox{with}\quad\iRd{\omega_n^2\,w_{\gamma_n}^{2(p-1)}\,|x|^{-\gamma_n}}=1\,.
\]
We divide the proof in three steps.

%---------------------------------------------------------------------
\medskip\noindent\textbf{First step.} \emph{For any $n\in\N$ large enough, $\omega_n$ belongs to $\dot{\mathrm H}^1(\R^d)$ and satisfies}
\be{eq: laplaciano-deriv-ang}
-\,\Delta{\omega_n}+p\,\frac{w_{\gamma_n}^{p-1}}{|x|^{\gamma_n}}\,\omega_n=(2\,p-1)\,\frac{w_{\gamma_n}^{2(p-1)}}{|x|^{\gamma_n}}\,\omega_n\quad\mbox{in}\quad\R^d\,.
\ee
%---------------------------------------------------------------------
The validity of~\eqref{eq: laplaciano-deriv-ang} can be proved just by plugging $\varphi_n(t,x)=\varphi\(e^{-t\kern0.5pt\mathsf A_n}\,x\)$ as a test function in the weak formulation of~\eqref{prob: Eul-Lag-one}, where $\varphi\in\mathcal{D}(\R^d)$, and change variables to get
\[\label{eq: laplaciano-deriv-ang-prova-1}
-\iRd{w_{\gamma_n}\!\(e^{t\kern0.5pt\mathsf A_n}\,x\)\Delta\varphi}+\iRd{\kern-2pt\tfrac{\(w_{\gamma_n}\(e^{t\kern0.5pt\mathsf A_n}\,x\)\)^p}{|x|^{\gamma_n}}\,\varphi}=\iRd{\kern-2pt\tfrac{\(w_{\gamma_n}\(e^{t\kern0.5pt\mathsf A_n}\,x\)\)^{2p-1}}{|x|^{\gamma_n}}\,\varphi}\,.
\]
Taking the derivative w.r.t.~$t$ at $t=0$ proves the identity. Note that, from the proof of Lemma \ref{lem-rego-calpha}, for $n$ large enough $\omega_n\in\mathrm H^1_{\mathrm{loc}}(\R^d)$ and~\eqref{eq: laplaciano-deriv-ang} holds in the $\mathrm H^1$ weak sense. We can now draw some consequences.

Let us multiply~\eqref{eq: laplaciano-deriv-ang} by the test function $\varphi=\xi_R\,\omega_n$ and integrate by parts, where $\xi_R(x):=\xi(x/R)$ and $\xi$ is a smooth cut-off function:
\begin{multline}\label{eq: energia-deriv-ang-prova-1}
\iRd{\xi_R\left|\nabla{\omega_n}\right|^2}-\frac12\iRd{\Delta\xi_R\,\omega_n^2}+p\iRd{\xi_R\,\omega_n^2\,\frac{w_{\gamma_n}^{p-1}}{|x|^{\gamma_n}}}\\
=(2\,p-1)\iRd{\xi_R\,\omega_n^2\,\frac{w_{\gamma_n}^{2(p-1)}}{|x|^{\gamma_n}}}\,.
\end{multline}
Since $\{w_{\gamma_n}\}$ converges up to the extraction of subsequences in $\dot{\mathrm H}^1(\R^d)$, by its definition we have that $\{\omega_n\}$ is also relatively compact in $\mathrm L^2_2(\R^d)$ and with the estimate
\[\label{eq: energia-deriv-ang-prova-2}
\left|\iRd{\Delta{\xi_R}\,\omega_n^2}\right|\le\frac{\|\Delta\xi\|_\infty}{R^2}\int_{B_{2R}\setminus B_R}\omega_n^2\,dx\le4\,\|\Delta\xi\|_\infty\int_{B_{2R}\setminus B_R}\frac{\omega_n^2}{|x|^2}\,dx\,,
\]
we obtain that $\lim_{R\to\infty}\left|\iRd{\Delta{\xi_R}\,\omega_n^2}\right|=0$. Since $\iRd{w_{\gamma_n}^{2(p-1)}\,|x|^{-\gamma_n}\,\omega_n^2}$ is uniformly bounded by Corollary~\ref{coro: convergenza-u-gamma}, we can then pass the limit as $R\to\infty$ in~\eqref{eq: energia-deriv-ang-prova-1} and obtain
\be{eq: energia-deriv-ang}
\iRd{\left|\nabla{\omega_n}\right|^2}+p\iRd{\frac{w_{\gamma_n}^{p-1}}{|x|^{\gamma_n}}\,\omega_n^2}=(2\,p-1)\iRd{\frac{w_{\gamma_n}^{2(p-1)}}{|x|^{\gamma_n}}\,\omega_n^2}\,.
\ee

For $n$ large enough, we still have to show that $\omega_n\in\dot{\mathrm H}^1(\R^d)$, namely that there exists a sequence $\{\varphi_k\}\subset\mathcal{D}(\R^d)$ such that \mbox{$\lim_{k\to\infty}\nrm{\nabla{\omega_n}-\nabla{\varphi_k}}2=0$}. Actually, this is a direct consequence of the fact that $|\nabla \omega_n| \in \mathrm L^2(\R^d)$ and $\omega_n\in \mathrm L^2_2(\R^d)$.

Using once again a cut-off argument, we find that
\[\label{eq: medianulla-deriv-ang-prova-2}
\left|\iRd{\xi_R\,\nabla_{\mathsf A}\!\(\frac{w_{\gamma_n}^{2p}}{|x|^{\gamma_n}}\)}\right| \!=\! \left|\iRd{\(\nabla_{\mathsf A}\,\xi_R\) \frac{w_{\gamma_n}^{2p}}{|x|^{\gamma_n}}}\right| \! \le\| \nabla_{\mathsf A} \xi\|_\infty\int_{B_{2R}\setminus B_R}\frac{w_{\gamma_n}^{2p}}{|x|^{\gamma_n}}\,dx\,.
\]
Because of the Cauchy-Schwarz estimate
\[\label{eq: medianulla-deriv-ang-prova-3}
\(\iRd{\frac{w_{\gamma_n}^{2p-1}\,|\omega_n|}{|x|^{\gamma_n}}}\)^2\le\iRd{\frac{w_{\gamma_n}^{2(p-1)}}{|x|^{\gamma_n}}\,\omega_n^2}\iRd{\frac{w_{\gamma_n}^{2p}}{|x|^{\gamma_n}}}\,,
\]
we may let $R\to\infty$ and obtain
\be{eq: medianulla-deriv-ang}
\iRd{\nabla_{\mathsf A}\!\(w_{\gamma_n}^{2p}\,|x|^{-{\gamma_n}}\)}=2\,p\iRd{\omega_n\,w_{\gamma_n}^{2p-1}\,|x|^{-\gamma_n}}=0\,.
\ee

%---------------------------------------------------------------------
\medskip\noindent\textbf{Second step.} \emph{As $n\to\infty$, up to the extraction of a subsequen\-ce, $\{\omega_n\}$ converges weakly in $\dot{\mathrm H}^1(\R^d)$ to a nontrivial function $\omega$ such that}
\be{eq: func-HP-limite}
\iRd{\left|\nabla{\omega}\right|^2 dx+p\int_{\R^d}w_0^{p-1}\,\omega^2}\le(2\,p-1)\iRd{w_0^{2(p-1)}\,\omega^2}\,,
\ee
\be{eq: func-HP-limite-constraints}
\iRd{w_0^{2p-1}\,\omega}=0\quad\mbox{and}\quad\iRd{\frac{x_i}{|x|^2}\(w_0^p-w_0^{2p-1}\)\omega}=0\quad\forall\,i=1,\,2,\ldots d\,.
\ee
%---------------------------------------------------------------------
Recall that $\mathsf A_n$ has been normalized by \hbox{$\iRd{\omega_n^2\,w_{\gamma_n}^{2(p-1)}\,|x|^{-\gamma_n}}=1$}. According to~\eqref{eq: energia-deriv-ang}, the sequence $\{\omega_n\}$ is bounded in $\dot{\mathrm H}^1(\R^d)$ and, up to subsequences, converges in $\dot{\mathrm H}^1(\R^d)$ and pointwise to some function $\omega$. By the Sobolev inequality, we infer that $\{\omega_n^2\}$ converges weakly in $\mathrm L^{d/(d-2)}(\R^d)$ to $\omega^2$, up to subsequences. The sequence $\{w_{\gamma_n}^{2(p-1)}|x|^{-\gamma_n}\}$ converges strongly in $\mathrm L^{d/2}(\R^d)$ to $w_0^{2(p-1)}$ as a consequence of Corollary~\ref{coro: convergenza-u-gamma}, Proposition~\ref{lem: minimizer} and of the dominated convergence theorem. We therefore deduce that
\[\label{eq: energia-deriv-ang-massa}
1=\lim_{n\to\infty}\iRd{\frac{w_{\gamma_n}^{2(p-1)}}{|x|^{\gamma_n}}\,\omega_n^2}=\iRd{w_0^{2(p-1)}\,\omega^2}\,.
\]
This also proves~\eqref{eq: func-HP-limite} by weak lower semi-continuity of the r.h.s.

We then observe that the sequences $\{w_{\gamma_n}^{p-1}\,|x|^{-\gamma_n/2}\,\omega_n \}$ and $\{w_{\gamma_n}^p\,|x|^{-\gamma_n/2}\}$ converge strongly in $\mathrm L^2(\R^d)$ to $w_0^{p-1}\,\omega$ and $w_0^p$, respectively. This allows to pass to the limit in~\eqref{eq: medianulla-deriv-ang} and proves that $\iRd{w_0^{2p-1}\,\omega}=0$.

Let us take $\mathsf e \in\mathbb S^{d-1}$ and consider the \emph{directional derivative} \hbox{$v_n:=\mathsf e\cdot\nabla w_{\gamma_n}$}. Proceeding similarly to the proof of~\eqref{eq: laplaciano-deriv-ang}, we get that $v_n\in\dot{\mathrm H}^1(\R^d)$ solves exactly the same equation as $\omega_n$ except for a forcing term arising from the directional derivative of the weight $|x|^{-\gamma_n}$:
\[\label{eq: laplaciano-deriv-direz}
-\,\Delta{v_n}+p\,\frac{w_{\gamma_n}^{p-1}}{|x|^{\gamma_n}}\,v_n=(2\,p-1)\,\frac{w_{\gamma_n}^{2(p-1)}}{|x|^{\gamma_n}}\,v_n+\gamma_n\,\frac{x\cdot \mathsf e}{|x|^{\gamma_n+2}}\(w_{\gamma_n}^p-w_{\gamma_n}^{2p-1}\)
\]
in $\R^d$. Using $\omega_n$ as a test function and subtracting the identity obtained by taking $v_n$ as a test function in ~\eqref{eq: laplaciano-deriv-ang}, we end up with the identity
\[\label{eq: energia-deriv-ang-crucial}
\iRd{\frac{x\cdot \mathsf e}{|x|^{\gamma_n+2}}\(w_{\gamma_n}^p-w_{\gamma_n}^{2p-1}\)\omega_n}=0\,.
\]
Since $\{\omega_n\}$ converges to $\omega$ weakly in $\mathrm L^2_2(\R^d)$ and $\{x\cdot \mathsf e\,{|x|^{-\gamma_n}}\big(w_{\gamma_n}^p-w_{\gamma_n}^{2p-1}\big)\}$ converges to $x\cdot \mathsf e\,(w_0^p-w_0^{2p-1})$ strongly in $\mathrm L^2_2(\R^d)$, we may pass to the limit as $n\to\infty$. This proves that $\iRd{\frac{x\cdot \mathsf e}{|x|^2}\,(w_0^p-w_0^{2p-1})\,\omega}=0$ for any $\mathsf e \in \mathbb S^{d-1}$ and concludes the proof of~\eqref{eq: func-HP-limite-constraints}.

%---------------------------------------------------------------------
\medskip\noindent\textbf{Third step.} \emph{Conclusion of the proof: getting the contradiction.}
%---------------------------------------------------------------------

From Lemma~\ref{lem: hardy-poinc},~\eqref{eq: func-HP-limite} and the left-hand identity in~\eqref{eq: func-HP-limite-constraints}, we deduce that we have equality in~\eqref{eq: func-HP} if~$\omega$ is the limit as $n\to\infty$ of $\{\omega_n\}$. Hence $\omega=a\cdot\nabla w_0$ for some $a=(a_j)_{j=1}^d\in\R^d\setminus\{0\}$. From the right-hand identity in~\eqref{eq: func-HP-limite-constraints}, we infer that
\[\label{eq: func-HP-limite-ortogonal-w-1}
\sum_{j=1}^da_j\iRd{\tfrac{x_i}{|x|^2}\(w_0^p-w_0^{2p-1}\)(w_0)_{x_j}}=0\quad\forall\,i=1,\,2,\ldots d\,,
\]
that is
\[\label{eq: func-HP-limite-ortogonal-w-2}
\sum_{j=1}^da_j\iRd{\tfrac{x_i}{|x|^2}\(\tfrac{w_0^{p+1}}{p+1}-\tfrac{w_0^{2p}}{2\,p}\)_{x_j}}=0\quad\forall\,i=1,\,2,\ldots d\,.
\]
Integrating by parts, we obtain:
\be{eq: func-HP-limite-ortogonal-w-3}
\sum_{j=1}^da_j\iRd{\(\tfrac{\delta_{ij}}{|x|^2}-2\,\tfrac{x_i\,x_j}{|x|^4}\)\(\tfrac{w_0^{2p}}{2\,p}-\tfrac{w_0^{p+1}}{p+1}\)}=0\quad\forall\,i=1,\,2,\ldots d\,.
\ee
Since $w_0$ is radial, these integrals are identically zero for all $j\neq i$, so that \eqref{eq: func-HP-limite-ortogonal-w-3} reduces to
\[\label{eq: func-HP-limite-ortogonal-w-4}
 \sum_{i=1}^d a_i\iRd{\(\tfrac1{|x|^2}-2\,\tfrac{x_i^2}{|x|^4}\)\(\tfrac{w_0^{2p}}{2\,p}-\tfrac{w_0^{p+1}}{p+1}\)}=0\quad\forall\,i=1,\,2,\ldots d\,.
\]
The radiality of $w_0$ also implies that the integral is independent of $i$:
\[\label{eq: func-HP-limite-ortogonal-w-5}
\begin{aligned}
&\iRd{\(\tfrac1{|x|^2}-2\,\tfrac{x_i^2}{|x|^4}\)\(\tfrac{w_0^{2p}}{2\,p}-\tfrac{w_0^{p+1}}{p+1}\)}\
=&\tfrac{d-2}d\iRd{\tfrac1{|x|^2}\(\tfrac{w_0^{2p}}{2\,p}-\tfrac{w_0^{p+1}}{p+1}\)}\,.
\end{aligned}
\]
Since $a\not=0$, we deduce that $\iRd{\tfrac1{|x|^2}\,\big(\tfrac{w_0^{2p}}{2\,p}-\tfrac{w_0^{p+1}}{p+1}\big)}=0$. By arguing as in the proof of Proposition~\ref{lem: minimizer}, with $\ell(r)$ replaced by $1/r^2$, we obtain that the integral is positive, a contradiction.
\qed

%%%%%%%%%%%%%%%%%%%%%%%%%%%%%%%%%%%%%%%%%%%%%%%%%%%%%%%%%%%%%%%%%%%%%%
%%%%%%%%%%%%%%%%%%%%%%%%%%%%%%%%%%%%%%%%%%%%%%%%%%%%%%%%%%%%%%%%%%%%%%
\begin{acknowledgements}
J.D.~has been supported by the ANR projects \emph{NoNAP}, \emph{STAB} and \emph{Kibord}. B.N.~has been supported by the ANR project \emph{STAB}. M.M.~thanks Alessandro Zilio for a very helpful discussion concerning the regularity theory for the Euler-Lagrange equation. M.M.~has been partially funded by the ``Universit\`a Italo-Francese / Universit\'e Franco-Italienne'' (Bando Vinci 2013). The authors thank Guillaume Carlier for pointing them Ref.~\cite{AB}. They thank the referees for their careful reading which helped them to improve the manuscript.

\smallskip\noindent {\sl\small\copyright~2016 by the authors. This paper may be reproduced, in its entirety, for non-commercial purposes.}
\end{acknowledgements}
%%%%%%%%%%%%%%%%%%%%%%%%%%%%%%%%%%%%%%%%%%%%%%%%%%%%%%%%%%%%%%%%%%%%%%
%%%%%%%%%%%%%%%%%%%%%%%%%%%%%%%%%%%%%%%%%%%%%%%%%%%%%%%%%%%%%%%%%%%%%%
%\bibliographystyle{spmpsci}
%\bibliography{DMN}
%\end{document}
%%%%%%%%%%%%%%%%%%%%%%%%%%%%%%%%%%%%%%%%%%%%%%%%%%%%%%%%%%%%%%%%%%%%%%
%%%%%%%%%%%%%%%%%%%%%%%%%%%%%%%%%%%%%%%%%%%%%%%%%%%%%%%%%%%%%%%%%%%%%%

%%%%%%%%%%%%%%%%%%%%%%%%%%%%%%%%%%%%%%%%%%%%%%%%%%%%%%%%%%%%%%%%%%%%%%
%%%%%%%%%%%%%%%%%%%%%%%%%%%%%%%%%%%%%%%%%%%%%%%%%%%%%%%%%%%%%%%%%%%%%%
\end{document}